\documentclass[12pt]{article}   	
\usepackage{geometry}                		
\usepackage{graphicx}				
\usepackage{caption}
\usepackage{subcaption}
\usepackage{amssymb}
\usepackage{amsthm}
\usepackage{csquotes}
\usepackage{amsmath}
\usepackage{mathtools}
\usepackage{float}
\usepackage{physics}
\MakeOuterQuote{"}
\newtheorem*{theorem*}{Theorem}
\newtheorem{theorem}{Theorem}[section]
\newtheorem{definition}[theorem]{Definition}
\newtheorem{proposition}[theorem]{Proposition}
\newtheorem{lemma}[theorem]{Lemma}

\newtheorem{corollary}[theorem]{Corollary}

\newtheorem*{conjecture*}{Conjecture}
\title{Partial Regularity and Blowup for an Averaged Three-Dimensional Navier-Stokes Equation}
\author{Matei P. Coiculescu}
\begin{document}
\maketitle
\begin{abstract}
We prove two results that together strongly suggest that obtaining a positive answer to the Navier-Stokes global regularity question requires more than a refinement of partial regularity theory. First we prove that there exists a class of bilinear operators $\mathfrak{B}$, which contains the Euler bilinear operator $\mathcal{E}(u,v):=\frac{1}{2}\mathbb{P}(u\cdot\nabla v + v\cdot\nabla u)$, such that for any $B\in \mathfrak{B}$, $n\geq3$, $\alpha \in ((n+1)/4, (n+2)/4)$, and smooth solution $u$ of the pseudodifferential equation 
$\partial_t u +(-\Delta)^\alpha u +B(u,u)=0$ on $\mathbb{R}^n\times [0,T)$, we have that $u$ is also smooth at time $T$ away from a closed set of Hausdorff dimension at most $n+2-4\alpha$. Next we prove that, for the Euclidean space $\mathbb{R}^3$, there exists an operator $C(u,v)\in \mathfrak{B}$ that is an averaged version of $\mathcal{E}$, that formally allows the dissipation of energy by the "cancellation identity" $\langle C(u,u), u\rangle =0$, and whose corresponding pseudodifferential equation $\partial_t u +(-\Delta)^\alpha u +C(u,u)=0$ admits a solution that blows up in finite time for all $\alpha \in (0,5/4)$.
\end{abstract}

\section{Introduction}
The incompressible Navier-Stokes equations are a system of nonlinear partial differential equations that model the motion of an incompressible viscous fluid. In particular, if $u(x,t)$ is the velocity field of a fluid in $\mathbb{R}^3\times [0,T)$ dynamically changing in time $t$, with initial vector field $u_0$, external force $f$ and fluid pressure $p$, then the Navier-Stokes equations state that $u$ satisfies:
$$
\begin{gathered}
\partial_t u -\nu\Delta u+(u\cdot\nabla)u+\nabla p=f\\
\textrm{div } u=0\\
u(t=0)=u_0
\end{gathered}
$$
The parameter $\nu$ determines the viscosity of the fluid. Since its value has no influence on our work, we may assume that the viscosity $\nu=1$. We also consider the Navier-Stokes equations with "fractional dissipation". In particular let $\alpha\geq 0$ and consider the pseudodifferential operator $(-\Delta)^\alpha$ whose Fourier symbol is $\abs{\xi}^{2\alpha}$. Whenever $\alpha$ equals a positive integer $k$, the operator coincides, up to a sign, with the composition of the classical Laplacian $(-\Delta)$ with itself $k$ times. The unforced $\alpha$-dissipative Navier-Stokes system in $\mathbb{R}^3$ is the following system of pseudodifferential equations:
\begin{equation}
\label{FRACNS}
\begin{gathered}
\partial_t u+(-\Delta)^\alpha u+ (u\cdot \nabla)u +\nabla p = 0 \\
\textrm{div } u=0 \\
u(t=0)= u_0
\end{gathered}
\end{equation}
The $\alpha$-dissipative Navier-Stokes system is typically called hyperdissipative if $\alpha>1$ and hypodissipative if $\alpha<1$. The case when $\alpha=1$ coincides with the classical Navier-Stokes equations. For $\alpha\geq 5/4$ and for any smooth function $u_0$ which decays sufficiently fast at infinity, it is known that the system in Equation (\ref{FRACNS}) has a classical global-in-time solution. In the case of Equation (\ref{FRACNS}) on $\mathbb{R}^n$, classical global-in-time solutions are guaranteed whenever $\alpha >(n+2)/4$ and the initial data is smooth and has sufficient decay. We offer a proof of the latter fact in Section 5.6. Global regularity in $\mathbb{R}^3$ is also known for an open set of initial data if $\alpha$ is slightly less than $5/4$, see \cite{CH}. Generally, when $\alpha<5/4$, the global existence of classical solutions is still a major open question \cite{F}.

One approach towards proving global regularity is to show that the set where singularities of the vector field develop is small in some sense. In particular, one might want to find a bound on the Hausdorff dimension of such a set. Finding such bounds is within the scope of partial regularity theory. We now provide a short summary of the partial regularity theory for the classical Navier-Stokes equations. 

Leray showed in \cite{L} that the set of singular times has $1/2$-dimensional Hausdorff measure equal to zero. Scheffer wrote several papers, beginning with \cite{S}, bounding the Hausdorff dimension of the singular set in space at the the time of first blowup. The best known partial regularity result was obtained by Caffarelli, Kohn, and Nirenberg in \cite{CKN} for a particular class of weak solutions called suitable weak solutions. Let $u$ be a suitable weak solution on any open subset of spacetime $\mathbb{R}^3\times\mathbb{R}$. If we define the singular set (in space-time) as
$$S= \{ (x,t) : u \not\in L^\infty \textrm{ in any neighborhood of } (x,t)\}$$
then the $1$-dimensional Hausdorff measure of $S$ is equal to zero. Away from the set $S$, the vector field $u$ is bounded, which  implies that $u$ is smooth in the spatial variable by the Serrin higher regularity theorem, see \cite{SE} for more details. The bound on the Hausdorff dimension obtained by Caffarelli, Kohn, and Nirenberg in 1982 has not been improved. 

The work of Colombo, De Lellis, and Massaccesi in \cite{CDM} provides the optimal extension of the Caffarelli-Kohn-Nirenberg theory, dealing definitively with the hyperdissipative case. Tang and Yu in \cite{TY} and the erratum \cite{TY1} were the first to analyze the hypodissipative case $3/4<\alpha <1$. Kwon and Ozanski in \cite{KO} further sharpened the work done in \cite{TY}. Additionally, the only partial regularity result that we know of for non-stationary higher-dimensional Navier-Stokes equations is found in \cite{DG}, where Dong and Gu prove a result analogous to the one in \cite{CKN} in four dimensions. However, we cannot see a way to transfer these techniques to the setting of our interest. 

Our present work is inspired by the paper \cite{KP} of Katz and Pavlovic. The authors of \cite{KP} prove a partial regularity theorem like Scheffer's in \cite{S} by using Fourier analysis to reduce the problem to the study of an infinite-dimensional system of ordinary differential equations. The fluid vector field is decomposed into wavelet energy packets by localizing in both physical and frequency space in a way compatible with the Heisenberg Uncertainty Principle, and the dynamics of the wavelet coefficients can be considered a system of ordinary differential equations. One of the advantages of their approach when compared to the work of Scheffer or Caffarelli, Kohn, and Nirenberg is that the local energy inequality is unnecessary for the proof. Ozanski's work in \cite{O}, which clarifies and extends the Fourier analytic method of Katz and Pavlovic, gets a partial regularity result for Leray-Hopf weak solutions of the hyperdissipative Navier-Stokes equations in $\mathbb{R}^3$ as well as a bound on the box-counting dimension of the singular set.

We now briefly discuss some work done in the "negative" direction of the global regularity problem.
Besides numerical studies like \cite{H} that suggest the development of a finite-time singularity for the Navier-Stokes equations, the paper \cite{T} by Tao provides evidence that at the very least, attacking the problem of global regularity using abstract methods depending on basic $L^p$ bounds for the nonlinearity is a direction certain to fail. Let $\mathbb{P}$ denote the Leray projection on divergence-free vector fields (see Section 1.6 in \cite{TS}). Tao constructs a nonlinearity $C(u,v)$ that is an averaged version of the Euler bilinear operator
$$\mathcal{E}(u,v)=\frac{1}{2}\mathbb{P}\bigg((u\cdot \nabla) v+(v\cdot \nabla) u\bigg),$$
satisfies the cancellation identity and various "harmonic-type" estimates such as
$$\|C(u,v) \|_{p} \leq K\big(\|u\|_{q}\|\nabla v\|_{r}+\|v\|_{q}\|\nabla u\|_{r}\big)$$
where $1/p = 1/q+1/r$ and $1< p,q,r< \infty$, and whose associated partial differential equation admits a finite-time singularity. Here and in the sequel, that an operator $C$ is an "averaged version of $\mathcal{E}$" means that:
$$\langle C(u,v) , w\rangle =$$ \begin{equation} \label{AVERAGE}= \int_\Omega \langle \mathcal{E}(m_{1,\omega}\textrm{Rot}_{R_{1,\omega}}\textrm{Dil}_{\lambda_{1,\omega}} u, m_{2,\omega}\textrm{Rot}_{R_{2,\omega}}\textrm{Dil}_{\lambda_{2,\omega}} v), m_{3,\omega}\textrm{Rot}_{R_{3,\omega}}\textrm{Dil}_{\lambda_{3,\omega}} w\rangle d\mu(\omega)\end{equation}
for all Schwartz vector fields $u,v,w$. Here, $\omega$ is a variable in the probability space $(\Omega, \mu)$, and the maps $R_{i,\omega}: \Omega \to SO(3), \lambda_{i,\omega}: \Omega \to [C^{-1}, C],$ and $m_{i,\omega}: \Omega \to \mathfrak{M}_0$ are measurable maps. We use the notation $\mathfrak{M}_0$ for the space of all Fourier multipliers $m$ of order $0$ with Schwartz symbols equipped with the topology (hence also the $\sigma$-algebra) generated by the Schwartz seminorms. The operators Dil and Rot are respectively the dilation and rotation operator in $\mathbb{R}^3$. See \cite{T} for more details. Similarly, Scheffer proves in \cite{S2} that suitable weak solutions of the Navier-Stokes inequality admit interior singularities. Later in \cite{S3}, Scheffer shows that singularity formation can occur on a fractal set of any dimension less than one, thereby proving that the method used in \cite{CKN} to achieve partial regularity is sharp.

Our main result suggests that improving partial regularity results for the Navier-Stokes equations is not a path to answering the global regularity question, unless one uses more specific structural properties of the Navier-Stokes nonlinearity. We denote the parabolic cylinder of radius $r$ at $(x,t)$ by:
$$\mathcal{P}_r(x,t):= \{ (y,s) : \abs{y-x}< r \textrm{ and } t-r^2< s < t\}$$
Here and in the sequel we denote the closed singular set at blowup time $T$ by
$$S_T:= \{ x\in\mathbb{R}^n : \forall r>0, u\not\in  L^\infty_t C^\infty_x(\mathcal{P}_r(x,T))\}.$$
Our main theorem provides an upper bound on the Hausdorff dimension of $S_T$ for blowup solutions of an averaged three-dimensional $\alpha$-dissipative Navier-Stokes Equation. In particular, we prove:
\begin{theorem}
\label{MAINMAIN}
Let $1<\alpha<5/4$ be arbitrary.
There exists a symmetric bilinear operator $C$ that is an averaged version of $\mathcal{E}$, that obeys the cancellation identity
$$\langle C(u,u), u\rangle =0$$
for all $u\in H^{10}(\mathbb{R}^3)$ that are divergence-free in the distributional sense, and such that, for some Schwartz divergence-free initial vector field $u_0$, there is no global-in-time solution to
$$\begin{gathered} \partial_t u +(-\Delta)^\alpha u +C(u,u) +\nabla p =0\\
u(t=0)=u_0 \quad \quad \textrm{div } u=0.\end{gathered}
$$
Moreover, if $T$ is the time of first blowup, then the closed set $S_T$ has Hausdorff dimension at most $5-4\alpha$.
\end{theorem}

The partial regularity portion of Theorem \ref{MAINMAIN} follows from a more general partial regularity theorem for a family of pseudodifferential equations of "Navier-Stokes type". Then, proving Theorem \ref{MAINMAIN} amounts to finding an equation in the family admitting a finite-time singularity. The proof of our partial regularity theorem will be in Section 2, and we now describe it in more detail. First we define a certain class $\mathfrak{B}$ of bilinear operators.

\begin{definition}
\label{AMEN}
We call a symmetric bilinear operator $B$ $\textbf{amenable}$ if and only if:
\begin{enumerate}
\item $B$ is defined for Schwartz vector fields
\item $B$ satisfies the Cancellation Identity:
\begin{equation}\label{Cancel}
\langle \mathbb{P}B(u,u), u\rangle=0 \quad \textrm{ for divergence-free } u
\end{equation}
\item The linear maps $B^1_v(u):= \mathbb{P}B(u,v), B^2_v(u):= \mathbb{P}B(v,u)$ are pseudodifferential operators in the class $OPS^m_{1,1}$ for some real number $m$ (the definition of this class may be found in Section 5.1)
\item For all $0<\gamma<3$ and for all $1\leq p_1, p_2, r\leq \infty$ satisfying
$$\frac{1}{p_1}+\frac{1}{p_2}+\frac{\gamma}{3} = \frac{1}{r}$$
we have for some constant $K$ depending only on $\gamma$, $p_1$, $p_2$, and $r$: \begin{equation}
\label{amenineq}\| \mathbb{P}B(u,v) \|_{L^r} \leq K \| u\|_{L^{p_1}}(\|v\|_{L^{p_2}} +\| \nabla v\|_{L^{p_2}})\end{equation} for all $u\in L^{p_1}$ and $v\in W^{1,p_2}$ with support of diameter at most $2^{100}$ (note that the endpoint cases of integrability are included).
\item Scaling: There exists a number $\lambda>1$ such that for any $k\in \mathbb{Z}$:
$$B(u(\lambda^k y), u(\lambda^k y))(x) = \lambda^k B(u(y),u(y))(\lambda^k x)$$
\end{enumerate}
\end{definition}
The class $\mathfrak{B}$ is the collection of all amenable operators. Throughout the rest of our work, we may alternately say that $B\in \mathfrak{B}$ or that $B$ is amenable. Our partial regularity theorem is:
\begin{theorem}
\label{PRMAINMAIN}
Let $n\geq 3$ and let $(n+1)/4<\alpha<(n+2)/4$. Let $B$ be an amenable bilinear operator, and let $u_0$ be a Schwartz vector field.
Suppose $u(t)$ is a smooth vector field on $\mathbb{R}^n\times [0,T) \to \mathbb{R}^n$ and suppose there exists a smooth function $p$ on $\mathbb{R}^n\times [0,T)$ such that:
$$
\begin{gathered}
\partial_t u +(-\Delta )^\alpha u+B(u,u)+\nabla p=0\\
\textrm{div } u=0\\
u(t=0)=u_0\end{gathered}
$$
If $T$ is the time of first blowup, then the closed set $S_T$ has Hausdorff dimension at most $n+2-4\alpha$.
\end{theorem}

One corollary of Theorem \ref{PRMAINMAIN} is a partial regularity result for the $\alpha$-dissipative Navier-Stokes equations for any dimension $n\geq 3$ whenever $\alpha>(n+1)/4$, which, since it is an improvement on what is already known, we state separately: 
\begin{corollary}
\label{CORAL}
Let $n\geq 3$ and let $(n+1)/4<\alpha<(n+2)/4$. Let $u_0$ be a Schwartz vector field.
Suppose $u(t)$ is a smooth vector field on $\mathbb{R}^n\times [0,T) \to \mathbb{R}^n$ and suppose there exists a smooth function $p$ on $\mathbb{R}^n\times [0,T)$ such that:
$$
\begin{gathered}
\partial_t u +(-\Delta )^\alpha u+(u\cdot\nabla)u+\nabla p=0\\
\textrm{div } u=0\\
u(t=0)=u_0\end{gathered}
$$
If $T$ is the time of first blowup, then the closed set $S_T$ has Hausdorff dimension at most $n+2-4\alpha$.
\end{corollary}

Previously, there was no known partial regularity result for the non-stationary $\alpha$-dissipative Navier-Stokes equations in $\mathbb{R}^n$ with spatial dimension $n\geq 5$.  By performing a highly technical analysis of pseudodifferential operators, we get a partial regularity result that also holds for a large family of pseudodifferential equations in addition to the fractionally dissipative Navier-Stokes equations. In particular, this allows us to apply our partial regularity theorem to a pseudodifferential equation admitting a solution blowing up in finite time. Thus, while the conditions placed on an amenable operator suffice to prove a partial regularity result for its associated pseudodifferential equation, they are insufficient to prove regularity for all time. We note that while previous results prove partial regularity of the solution in space-time, our result only concludes that the solution is "partially regular" in space at the blowup time. Thus, although our technique does not require the local energy inequality (and this is one reason why we can generalize to equations arising from amenable bilinear operators) our theorem has a weaker conclusion and does not prove a space-time partial regularity result.

Our paper is organized as follows. In Section 2 we prove Theorem \ref{PRMAINMAIN}. In particular, given the vector field solution $u$, we construct a set $E$ of Hausdorff dimension $n+2-4\alpha$ and prove that it contains the singular set $S_T$. The proof necessitates the use of the "barrier" construction also found in \cite{KP} and \cite{O}, which is the main obstacle to applying our method to the case of classical dissipation. The energy estimates used in Section 2 are proven in Section 3, where we test the pseudodifferential equation against "wavelet" localizations of the solution (localizations in both frequency and physical space). In Section 4, we show that Tao's method is flexible enough to construct a blowup solution to the pseudodifferential equation associated to some amenable bilinear operator, which, together with Theorem \ref{PRMAINMAIN}, proves Theorem \ref{MAINMAIN}. In Section 5, we prove some technical lemmas that are used throughout the paper. 

We would like to thank our advisor, Professor Camillo De Lellis, for suggesting the topic of this paper, for many helpful discussions, and for reviewing several drafts of our work. We also thank him for his constant support and encouragement throughout the process of writing this paper. We thank Vikram Giri for helpful discussions about \cite{KP}. We thank Wojciech Ozanski for his interest in our work, for many useful suggestions, and for pointing out an error in a previous version of this manuscript. We would also like to acknowledge the support of the National Science Foundation in the form of an NSF Graduate Research Fellowship and under Grant No. DMS-1946175.

\section{Partial Regularity}
We refer the reader to Section 5.1 whenever we mention pseudodifferential operators and to Section 5.2 whenever we mention the Littlewood-Paley partition of frequency space.
\subsection{Covering the Singular Set}
Let $\mathbb{P}$ denote the Leray projection. Let $\mathcal{F}$ denote the Fourier transform. We shall employ a Littlewood-Paley partition of frequency space. In particular, for $j\in \mathbb{Z}$ we have pseudodifferential operators $P_j$ with positive and smooth symbols $p_j(\xi)$ that are supported in $\frac{2}{3}2^j < \abs{\xi}<3\cdot 2^j$ and that also satisfy $p_j(\xi) = p_0(2^{-j} \xi)$ and 
\begin{equation}\sum_{j\in \mathbb{Z}} p_j(\xi) =1.\end{equation}
We also define $\tilde{P}_j:= \sum_{k=-2}^{2} P_{j+k}$ to be the sum of all Littlewood-Paley projections whose symbols' supports intersect the support of $p_j(\xi)$. Likewise we may define $\tilde{p}_j(\xi)$ as $\sum_{k=-2}^{k=2} p_{j+k}(\xi)$.
We can consider $P_j f$ as a combination of wavelets supported on cubes of sidelength $2^{-j}$. Keeping this heuristic in mind, we localize on cubes of sidelength almost equal to $2^{-j}$. That is, we fix an $\epsilon>0$ and localize within slightly larger cubes of sidelength $2^{-j(1-\epsilon)}$. Such a localization will allow us to commute cutoffs in frequency space with cutoffs in physical space, with only the addition of a negligible error. Since the Hausdorff dimension estimate is a closed condition, an approximate localization like this will be acceptable. Our localization is achieved with the use of particularly chosen bump functions that we describe below. See also \cite{KP} and \cite{O}. 

Let $Q$ be a cube of sidelength greater than $2^{-j(1-\epsilon)}$. We denote $\lambda Q$ to be the cube with the same center as $Q$ and with $\textrm{sidelength}(\lambda Q) =\lambda\cdot \textrm{sidelength}(Q)$. We define a bump function $1\geq \phi_{Q,j}\geq 0$ such that $\phi_{Q,j}=1$ on $Q$ and is zero outside of $(1+2^{-\epsilon j})Q$. We also require that for every multi-index $\alpha$, there is a constant $C_\alpha$ depending on $\alpha$ but independent of $Q$ and $j$ such that
\begin{equation}
\label{DBOUNDS}
\abs{D^\alpha \phi_{Q,j}}\leq C_\alpha 2^{\abs{\alpha}j(1-\epsilon)}
\end{equation}
Following \cite{KP}, we call these bump functions of type $j$. In line with our previous discussion, we consider $\phi_{Q,j}P_j$ as a projection onto a localized wavelet, and, if $Q$ is a cube with side length exactly equal to $2^{-j(1-\epsilon)}$, we denote the "wavelet coefficient" by
$$\| \phi_{Q,j}P_j f\|_2:= f_Q$$
and say that $Q$ is at level $j$. 
For convenience, we shall use the following notation for a particular dilation of a cube at level $j$:
$$Q_j := (1+2^{-\epsilon j})Q.$$
Also, if $Q$ is a cube at level $j$, we denote:
$$u_{\mathcal{N}^1(Q)}:= \| \phi_{Q_j,j} P_{j-2\leq k \leq j+2}u\|_{L^2},$$
which is a shorthand for the kinetic energy supported on nearby frequencies. We shall also use the notation $N(A)$ to denote the cardinality of a set $A$.

Throughout Section 2 we use various properties satisfied by the bump functions $\phi_{Q,j}$ and the Littlewood-Paley projections $P_j$. Most of these properties can be classified as "commutator estimates" that allow us to move bump functions across Littlewood-Paley projections and vice versa with only the addition of a negligible error. For the benefit of exposition, we defer the technical proofs of these lemmas to Section 5; however, we shall use the results in Section 5 freely here.

We shall now prove a general partial regularity theorem for systems of pseudodifferential equations. These systems will generalize the Navier-Stokes equations in the sense that they arise from the Stokes equations equipped with a nonlinearity of the form $B(u,u)$, where $B$ is an amenable operator in the sense of Definition \ref{AMEN}. Without loss of generality, we may assume in this section that the scaling term $\lambda$ in the fifth part of Definition \ref{AMEN} is equal to $2$. 

We denote the Leray Projection of $B$ by $\hat{B}:=\mathbb{P}B$ for the rest of this section. By the incompressibility condition on $u$ in the statement of the theorem, we can instead consider the following system of pseudodifferential equations:
\begin{equation}
\label{PEQN}
\begin{gathered}
\partial_t u +(-\Delta)^\alpha u +\hat{B}(u,u)=0\\
u(t=0)=u_0.
\end{gathered}
\end{equation}
Then we prove Theorem \ref{PRMAINMAIN}, which we reformulate as
\begin{theorem}
\label{PRTHM}
Let $n\geq 3$ and $(n+1)/4<\alpha < (n+2)/4$, let $B$ be an amenable bilinear operator, and let $u_0$ be a Schwartz vector field.
Suppose that $u(t)$ is a smooth vector field on $\mathbb{R}^n\times [0,T) \to \mathbb{R}^n$ that solves Equation (\ref{PEQN}).
If $T$ is the time of first blowup, then the closed set $S_T$ has Hausdorff dimension at most $n+2-4\alpha$.
\end{theorem}

The next subsection is devoted to proving Theorem \ref{PRTHM}. We remark that if we take
$$B(u,v) = \frac{1}{2}\bigg( (u\cdot \nabla)v +(v\cdot \nabla u)\bigg)$$
then we have the usual $\alpha$-dissipative Navier-Stokes equations in Equation (\ref{PEQN}). Subsequently, we get Corollary \ref{CORAL}.

We recall that our solution $u$ is smooth up to time $T$ and that the initial data $u_0$ is Schwartz. These are facts that we shall use frequently and usually without reference. Pairing Equation (\ref{PEQN}) with $u$ and using the cancellation identity in Equation (\ref{Cancel}), we get:
\begin{equation}\label{NEEDIT}\langle\partial_t u, u \rangle + \langle(-\Delta)^\alpha u , u \rangle =0.\end{equation}
Integrating in time then gives us what we shall refer to as the \textbf{energy dissipation law}:
\begin{equation}
\label{Conserv}
\| u(t)\|_2^2 = \|u_0\|_2^2 -2\int_0^t \|(-\Delta)^{\alpha/2} u\|_2^2 dt \leq \|u_0\|_2^2 \quad \forall t\in[0,T).\end{equation}
We shall examine the dynamics of the "wavelet coefficients" $u_Q:= \| \phi_{Q,j}P_j f\|_2$, where $\phi_{Q,j}$ is a bump function of type $j$ and $Q$ is a cube at level $j$. To this end, we pair Equation (\ref{PEQN}) with $P_j \phi_{Q,j}^2 P_j u$, and get, after simplification:
\begin{equation}
\label{COEFFODE}
\partial_t \bigg(\frac{1}{2} u_{Q}^2\bigg)= \quad \langle-\hat{B}(u,u),P_j \phi_{Q,j}^2 P_j u \rangle  -\langle(-\Delta)^{\alpha}u,P_j \phi_{Q,j}^2 P_j u \rangle. \end{equation}
We note that the "wavelet coefficients" $u_Q$ are continuous in time, see Lemma \ref{Integral1} and Corollary \ref{Integral2}. In Section 3, we prove two main estimates for the terms on the right-hand-side of Equation (\ref{COEFFODE}). We estimate the "dissipation" term by
\begin{equation}\label{BIGEST1}
\langle(-\Delta)^\alpha u , P_j \phi_{Q,j}^2 P_j u \rangle \quad \geq -K2^{-200j}+K2^{2\alpha j}u_Q^2 -K 2^{(2\alpha-\epsilon)j} u_{\mathcal{N}^1(Q)}^2.
\end{equation}
We estimate the "nonlinear" term by
$$\abs{ \langle-\hat{B}(u,u),P_j \phi_{Q,j}^2P_j u\rangle} \leq K2^{j(1+\frac{n\delta}{2}+\frac{n\gamma}{3})}u_{\mathcal{N}^1(Q)}u_{Q}+$$ $$+K\sum_{k=\delta j}^{j-10} 2^{\frac{nk}{2}+j +\frac{n\gamma j}{3}}u_{Q_k}u_{\mathcal{N}^1(Q)}u_{Q}+ K \sum_{k>j+10} 2^{nj(\frac{1}{2}-\frac{2\gamma}{3}) +k(n\gamma +1)}u_{Q}\|\phi_{Q_j,j}P_k u\|_2^2 + $$ \begin{equation}\label{BIGEST2}K2^{\frac{(n+2)j}{2}+\frac{n\gamma j}{3}}u_{Q}\|\phi_{Q_j,j}P_{j-10\leq k\leq j+10}u\|_{L^2}^2+K2^{-150j}.\end{equation}
We have used the following notation: for a cube $Q$ at level $j$ we denoted
$$k<j \Rightarrow Q_k := 2^{(j-k)(1-\epsilon)}Q,\quad k\geq j\Rightarrow Q_k = (1+2^{-\epsilon k})Q.$$ 
The estimates hold for all cubes $Q$ at level $j\geq 0$, and we shall now clarify the role of the symbol $K$.

Throughout the rest of our paper, the symbol $K$ will denote a constant appropriately chosen so that every inequality where it appears is satisfied. The choice of $K$ will be clear at each instance of its use, and it will only depend on the $L^2$ norm of the initial data, the blowup time $T$, the dimension $n\geq 3$, and the parameters $\epsilon, \alpha,$ and $\gamma$. Also, we introduce another parameter $\delta$ that depends on $\alpha$ and the dimension $n\geq 3$. The parameter $\delta$ will be used in Section 4 when we estimate the nonlinear term in Equation (\ref{COEFFODE}). Given $n\geq 3$ and $\alpha>(n+1)/4$, we shall carefully choose the parameters $\delta$, $\gamma$, and $\epsilon$ so that our argument closes.

We now describe the construction of a collection of cubes, away from which we have "critical regularity", as in \cite{KP} and \cite{O}. We call a cube $Q$ at level $j$ \textbf{bad} if and only if:
$$\int_0^T \int \sum_{k\geq j} 2^{2\alpha k}\abs{\phi_{Q_j,j}P_{k} u}^2 dxdt \geq  2^{-(n+2-4\alpha)j-\epsilon j}.$$
We shall consider the reverse inequality as the hypothesis of an $\epsilon$-regularity statement compatible with the scaling of the equations. In other words, 
$$(2^m)^{2\alpha -1}u((2^m)x, (2^m)^{2\alpha} t)$$
is a solution of Equation (\ref{PEQN}) for any $m\in \mathbb{Z}$, if $u(x,t)$ is a solution of the same equation. We define the set $M_j$ to be the union of cubes $Q$ for all bad cubes $Q$ at level $j$.
We first obtain a convenient cover of $M_j$:
\begin{proposition}\label{1cover1} There exists a constant $K$ depending only on $\epsilon$ so that 
for all $j\geq K$, there is a covering $\mathcal{C}_j$ of $M_j$ by cubes at level $j$ such that for some other constant $K$ independent of $j$:
$$N(\mathcal{C}_j ) \leq K 2^{(n+2-4\alpha)j+\epsilon j}.$$
\end{proposition}
\begin{proof}
Let $\mathcal{A}$ be the collection of cubes $2Q$ where $Q$ is bad and at level $j$. By the Vitali Covering Lemma (Lemma \ref{Vitali}), there are disjoint cubes $2Q'$ in $\mathcal{A}$ such that the collection of cubes $10Q'$ covers $M_j$. However, any cube of sidelength $10\cdot 2^{-j(1-\epsilon)}$ can be covered by $K$ cubes at level $j$ ($K$ depending on $\epsilon$ and $n\geq 3$ only), so we are led to define $\mathcal{C}_j$ to be the collection of the cubes at level $j$ that cover the cubes $10Q'$. Now since the cubes $Q'$ are bad with cubes $2Q'$ disjoint, the cubes $Q_j'$ are also disjoint. Indeed, by always choosing $j\geq 100/\epsilon$, each such cube $Q_j' = (1+2^{-\epsilon j})Q'\subset (3/2)Q'$ is strictly contained in the cube $2Q'$. By the energy dissipation law and Lemma \ref{finiteband1} we may conclude:
$$N(\mathcal{C}_j)2^{-(n+2-4\alpha)j-\epsilon j}\leq$$
$$\leq K \sum_{Q'} \sum_{k\geq j}2^{2\alpha k} \int_0^T\int \abs{\phi_{Q'_j,j}P_k u}^2 dxdt\leq K  \sum_{k\geq j}2^{2\alpha k}\int_0^T \int \sum_{Q'}\abs{\phi_{Q'_j,j}P_{k} u}^2dxdt \leq $$ $$\leq K   \sum_{k\geq j}2^{2\alpha k}\int_0^T \int_{\mathbb{R}^n}\abs{P_{k} u}^2 dxdt \leq K,$$
which shows that there are at most $K2^{(n+2-4\alpha)j+\epsilon j}$ cubes in the collection. 
\end{proof}

\subsection{Critical Regularity Away from Bad Cubes}
Note that the cube $Q$ in the following proposition may actually be bad. The only condition that we place on $Q$ is that some large dilations of $Q$ are not bad, and this will be sufficient for the following proposition.
\begin{proposition}
\label{PRCRITREG}
There exists $j_0$ sufficiently large so that for any cube $Q$ at level $j$ such that the cubes $Q_{k}$ are not bad for any $k\in [\delta j, j-10]$, we have for all $t\in [0,T]$:
$$u_Q(t)< 2^{-\frac{j}{2}(n+2-4\alpha+\epsilon)}.$$
\end{proposition}
\begin{proof}
The claim is true for sufficiently small $t\geq 0$ because our initial data is Schwartz. Let $t_0>0$ be the first time when the claim is not true. More explicitly, we assume that for any $j_0$ there exists $Q$ at some level $j$ such that the cubes $Q_{k}$ are not bad for all $k\in [\delta j, j-10]$ but $u_Q(t_0)\geq 2^{-\frac{j}{2}(n+2-4\alpha+\epsilon)}$. Let $t_1\in (0,t_0)$ be the largest time so that $u_Q(t_1)\leq \frac{1}{2}2^{-\frac{j}{2}(n+2-4\alpha+\epsilon)}$, so that 
\begin{equation}\label{CRCONTRA}\frac{1}{2}2^{-\frac{j}{2}(n+2-4\alpha+\epsilon)} \leq u_Q(t) \leq  2^{-\frac{j}{2}(n+2-4\alpha+\epsilon)}\end{equation}
for all $t\in (t_1, t_0)$. Recall that our estimate in Equation (\ref{BIGEST2}) states that for any $\delta>0$:
$$\abs{ \langle-\hat{B}(u,u),P_j \phi_{Q,j}^2P_j u\rangle} \leq K2^{j(1+\frac{n\delta}{2}+\frac{n\gamma}{3})}u_{\mathcal{N}^1(Q)}u_{Q}+$$ $$+K\sum_{k=\delta j}^{j-10} 2^{\frac{nk}{2}+j +\frac{n\gamma j}{3}}u_{Q_k}u_{\mathcal{N}^1(Q)}u_{Q'}+ K \sum_{k>j+10} 2^{nj(\frac{1}{2}-\frac{2\gamma}{3}) +k(n\gamma +1)}u_{Q'}\|\phi_{Q_j,j}P_k u\|_2^2 +$$ $$+K2^{\frac{(n+2)j}{2}+\frac{n\gamma j}{3}}u_{Q}\|\phi_{Q_j,j}P_{j-10\leq k\leq j+10}u\|_{L^2}^2+K2^{-150j}.$$
We thus conclude from Equation (\ref{COEFFODE}) and Equation (\ref{BIGEST1}) that
$$\frac{3}{4} 2^{-j(n+2-4\alpha+\epsilon)} \leq   u_Q^2(t_0)- u_Q^2(t_1) \leq \int_{t_1}^{t_0}\bigg( K2^{j(1+\frac{n\delta}{2}+\frac{n\gamma}{3})}u_{\mathcal{N}^1(Q)}u_{Q}+$$ $$+K\sum_{k=\delta j}^{j-10} 2^{\frac{nk}{2}+j +\frac{n\gamma j}{3}}u_{Q_k}u_{\mathcal{N}^1(Q)}u_{Q}+ K \sum_{k>j+10} 2^{nj(\frac{1}{2}-\frac{2\gamma}{3}) +k(n\gamma +1)}u_{Q}\|\phi_{Q_j,j}P_k u\|_2^2 +$$ $$+K2^{\frac{(n+2)j}{2}+\frac{n\gamma j}{3}}u_{Q}\|\phi_{Q_j,j}P_{j-10\leq k\leq j+10}u\|_{L^2}^2\bigg)dt- \int_{t_1}^{t_0}K2^{2\alpha j} u_Q^2dt +$$ \begin{equation}\label{BIGEST100}\int_{t_1}^{t_0} K 2^{(2\alpha-\epsilon)j}u_{\mathcal{N}^1(Q)}^2 +K2^{-150j} .\end{equation}
We observe that since $Q_{j-10}$ is not a bad cube we have, for some constant $K$ depending on $\alpha, \epsilon, n$:
\begin{equation}\label{NOTMILD}\int_0^T \int \sum_{k\geq j-10} 2^{2\alpha k} \abs{\phi_{Q_{j-10},j-10}P_k u}^2 dxdt \leq K 2^{-(n+2-4\alpha +\epsilon)j}\end{equation}
and in particular
\begin{equation}\label{TOOGOOD}\int_0^T u_Q^2 dt\leq\int_{0}^T u_{\mathcal{N}^{1}(Q)}^2dt \leq\int_{0}^T \| \phi_{Q_j,j} P_{j-10\leq k \leq j+10}\|_{L^2}^2dt \leq K2^{-(n+2-2\alpha+\epsilon)j}.\end{equation}
In addition, since $Q_{k}$ is not bad for any $k\in [\delta j, j-10]$, we have
\begin{equation}\label{TOOGOOD2}\int_0^T u_{Q_k}^2 dt \leq K2^{-(n+2-2\alpha+\epsilon)k}\end{equation}
for any $k\in [\delta j, j-10]$.
We shall reach a contradiction by appropriately bounding the terms in the estimate of the nonlinearity. First we observe by the Cauchy-Schwarz inequality:
$$2^{j(1+\frac{n\delta}{2}+\frac{n\gamma}{3})}\int_{t_1}^{t_0}u_{\mathcal{N}^1(Q)}u_{Q} \leq 2^{j(1+\frac{n\delta}{2}+\frac{n\gamma}{3})}\bigg(\int_{t_1}^{t_0}u_{\mathcal{N}^1(Q)}^2dt\bigg)^{1/2}\bigg(\int_{t_1}^{t_0}u_{Q}^2dt\bigg)^{1/2}.$$
Using Equation (\ref{TOOGOOD}), we have that
$$K2^{j(1+\frac{n\delta}{2}+\frac{n\gamma}{3})}\int_{t_1}^{t_0}u_{\mathcal{N}^1(Q)}u_{Q} \leq K2^{j(1+\frac{n\delta}{2}+\frac{n\gamma}{3}-(n+2-2\alpha+\epsilon))j}\leq K 2^{-(n+2-4\alpha+\epsilon)j} 2^{(1+\frac{n\delta}{2}+\frac{n\gamma}{3}-2\alpha)j}.$$
Since $\alpha>(n+1)/4$, we may choose $\delta<(n-1)/n$ and  $\gamma$ sufficiently small so that
$$1+\frac{n\delta}{2}+\frac{n\gamma}{3}-2\alpha < \frac{n\gamma}{3} +2\big(\frac{n+1}{4}-\alpha\big)<0.$$
Thus, choosing $j_0$ sufficiently large, we have for all $j\geq j_0$:
$$K2^{(1+\frac{n\delta}{2}+\frac{n\gamma}{3}-2\alpha)j} \leq \frac{1}{1000}.$$
Therefore
\begin{equation}\label{STEP100}K2^{j(1+\frac{n\delta}{2}+\frac{n\gamma}{3})}\int_{t_1}^{t_0}u_{\mathcal{N}^1(Q)}u_{Q} \leq \frac{1}{1000} 2^{-(n+2-4\alpha+\epsilon)j}.\end{equation}
We continue our analysis with the low-high terms. We observe by the Cauchy-Schwarz inequality and Equation (\ref{CRCONTRA}) that
$$2^{\frac{nk}{2}+j +\frac{n\gamma j}{3}}\int_{t_1}^{t_0}u_{Q_k}u_{\mathcal{N}^1(Q)}u_{Q}\leq 2^{\frac{nk}{2}+j+\frac{n\gamma j}{3}} 2^{-\frac{j}{2}(n+2-4\alpha+\epsilon)}\bigg(\int_{t_1}^{t_0} u_{Q_k}^2 dt\bigg)^{1/2} \bigg(\int_{t_1}^{t_0} u_{\mathcal{N}^1(Q)}^2dt\bigg)^{1/2}.$$
Thus, by Equation (\ref{TOOGOOD}) and Equation (\ref{TOOGOOD2}), we have
$$K2^{\frac{nk}{2}+j +\frac{n\gamma j}{3}}\int_{t_1}^{t_0}u_{Q_k}u_{\mathcal{N}^1(Q)}u_{Q}\leq K 2^{\frac{nk}{2}+j+\frac{n\gamma j}{3}} 2^{-\frac{j}{2}(n+2-4\alpha+\epsilon)} 2^{-(n+2-2\alpha+\epsilon)\frac{k}{2}}2^{-(n+2-2\alpha+\epsilon)\frac{j}{2}}.$$
Simplifying, we get
\begin{equation}\label{ESTIMATEY}K2^{\frac{nk}{2}+j +\frac{n\gamma j}{3}}\int_{t_1}^{t_0}u_{Q_k}u_{\mathcal{N}^1(Q)}u_{Q}\leq K2^{-j(n+2-4\alpha+\epsilon)}2^{-(2-2\alpha+\epsilon)\frac{k}{2}}2^{j+\frac{n\gamma j}{3}-\alpha j}.\end{equation}
Recall that $n\geq 3$ and $\alpha >(n+1)/4 \geq 1$, so we get (also assuming $0<\epsilon<2\alpha-2$) from Equation (\ref{ESTIMATEY}):
$$\sum_{k=\delta j}^{j-10}K2^{\frac{nk}{2}+j +\frac{n\gamma j}{3}}\int_{t_1}^{t_0}u_{Q_k}u_{\mathcal{N}^1(Q)}u_{Q}\leq K2^{-j(n+2-4\alpha+\epsilon)}2^{-(2-2\alpha+\epsilon)\frac{j}{2}}2^{j+\frac{n\gamma j}{3}-\alpha j}.$$
Assuming that $\gamma<3\epsilon/(2n)$ gets us
$$-(1-\alpha+\epsilon/2)+1+\frac{n\gamma}{3}-\alpha= n\gamma/3-\epsilon/2<0.$$
Then choosing $j_0$ large enough, we may conclude that
$$K2^{-(2-2\alpha+\epsilon)\frac{j}{2}}2^{j+\frac{n\gamma j}{3}-\alpha j}\leq \frac{1}{1000}$$
for all $j\geq j_0$. This allows us to conclude that
\begin{equation}\label{STEP101}\sum_{k=\delta j}^{j-10}K2^{\frac{nk}{2}+j +\frac{n\gamma j}{3}}\int_{t_1}^{t_0}u_{Q_k}u_{\mathcal{N}^1(Q)}u_{Q}\leq \frac{1}{1000}2^{-j(n+2-4\alpha+\epsilon)}.\end{equation}
We proceed to examine the high-high terms. By Equation (\ref{CRCONTRA}) and Equation (\ref{NOTMILD}) we have
$$\int_{t_1}^{t_0} K \sum_{k>j+10} 2^{nj(\frac{1}{2}-\frac{2\gamma}{3}) +k(n\gamma +1)}u_{Q}\|\phi_{Q_j,j}P_k u\|_2^2 dt \leq$$ $$\leq 2^{-(n+2-4\alpha+\epsilon)\frac{j}{2}}\int_{t_1}^{t_0} K \sum_{k>j+10} 2^{nj(\frac{1}{2}-\frac{2\gamma}{3}) +k(n\gamma +1)}\|\phi_{Q_j,j}P_k u\|_2^2 dt\leq $$
$$\leq K 2^{-(n+2-4\alpha+\epsilon)\frac{j}{2}}2^{\frac{nj}{2}-\frac{2\gamma n j}{3}}2^{-j(n+2-4\alpha+\epsilon)}\sum_{k>j+10} 2^{k(n\gamma+1-2\alpha)}.$$
Since $\alpha>1$, and since we can always assume that $\gamma < 1/n$, we have that the sum converges, yielding 
$$\int_{t_1}^{t_0} K \sum_{k>j+10} 2^{nj(\frac{1}{2}-\frac{2\gamma}{3}) +k(n\gamma +1)}u_{Q}\|\phi_{Q_j,j}P_k u\|_2^2 dt \leq$$
$$\leq K2^{-(n+2-4\alpha+\epsilon)j} 2^{j(\frac{n}{2}-\frac{2\gamma n}{3}-\frac{(n+2-4\alpha+\epsilon)}{2}+n\gamma+1-2\alpha)}.$$
However, choosing $\gamma < \frac{3\epsilon}{2n}$, we can guarantee that 
$$\frac{n}{2}-\frac{2\gamma n}{3}-\frac{(n+2-4\alpha+\epsilon)}{2}+n\gamma+1-2\alpha = \frac{n\gamma}{3}-\frac{\epsilon}{2}<0,$$
so that, for some $j_0$ large enough, we in fact have
\begin{equation}\label{STEP102}\int_{t_1}^{t_0} K \sum_{k>j+10} 2^{nj(\frac{1}{2}-\frac{2\gamma}{3}) +k(n\gamma +1)}u_{Q}\|\phi_{Q_j,j}P_k u\|_2^2 dt \leq \frac{1}{1000} 2^{-(n+2-4\alpha+\epsilon)j}.\end{equation}
The local terms are dealt with by using Equation (\ref{TOOGOOD}) and Equation (\ref{CRCONTRA}). First, we see that
$$\int_{t_1}^{t_0} K 2^{\frac{(n+2)j}{2}+\frac{n\gamma j}{3}} u_Q \| \phi_{Q_j,j} P_{j-10\leq k \leq j+10}\|_{L^2}^2 dt \leq$$ $$\leq K2^{\frac{(n+2)j}{2}+\frac{n\gamma j}{3}-2\alpha j}2^{-\frac{j}{2}(n+2-4\alpha+\epsilon)}2^{-(n+2-4\alpha+\epsilon)j}.$$
Then, choosing $\gamma < \frac{3\epsilon}{2n}$, we can guarantee that 
$$\frac{n+2}{2}+\frac{n\gamma}{3}-2\alpha-\frac{(n+2-4\alpha+\epsilon)}{2}= \frac{n\gamma}{3}-\frac{\epsilon}{2}<0,$$
so that, for some $j_0$ large enough, we in fact have
\begin{equation}\label{STEP103}\int_{t_1}^{t_0} K 2^{\frac{(n+2)j}{2}+\frac{n\gamma j}{3}} u_Q \| \phi_{Q_j,j} P_{j-10\leq k \leq j+10}\|_{L^2}^2 dt \leq \frac{1}{1000} 2^{-(n+2-4\alpha+\epsilon)j}.\end{equation}
Finally, by Equation (\ref{TOOGOOD}), there exists $j_0$ large enough so that 
\begin{equation}\label{STEP104}\int_{t_1}^{t_0} K 2^{(2\alpha-\epsilon)j}u_{\mathcal{N}^1(Q)}^2 dt \leq  2^{-\epsilon j} K 2^{-(n+2-4\alpha+\epsilon)j} \leq \frac{1}{1000} 2^{-(n+2-4\alpha+\epsilon)j}.\end{equation}
By Equation (\ref{BIGEST100}), Equation (\ref{STEP100}), Equation (\ref{STEP101}), Equation (\ref{STEP102}), Equation (\ref{STEP103}), and Equation (\ref{STEP104}), we conclude that
$$\frac{3}{4} 2^{-j(n+2-4\alpha+\epsilon)}\leq \frac{1}{100} 2^{-j(n+2-4\alpha+\epsilon)} -K \int_{t_1}^{t_0} 2^{2\alpha j}u_Q^2 dt\leq \frac{1}{100}2^{-j(n+2-4\alpha+\epsilon)},$$
which implies that $3/4\leq 1/100,$ a contradiction.
\end{proof}

\subsection{The Barrier}
We follow the construction of the barrier as presented by Ozanski in \cite{O}, and we attempt to use similar notation, wherever we can.
Recall that we denoted the union of all bad cubes at level $j$ by $M_j$ and we found a collection of cubes at level $j$, denoted $\mathcal{C}_j$, that covers $M_j$ and satisfies 
$$N(\mathcal{C}_j) \leq K 2^{(n+2-4\alpha+\epsilon)j}.$$

As in \cite{O}, we wish to find another (possibly larger) cover of $M_j$, $\mathcal{B}_j$, of cubes at level $j$ satisfying the \textbf{barrier property}: for all $x\not\in \cup_{Q\in \mathcal{B}_j} Q$, there exists $0<r<1$ such that $\partial( rQ_j(x))\cap \cup_{Q\in \mathcal{B}_k} Q = \emptyset$ for all $k\geq j$. We denoted the cube at level $j$ centered at $x$ by $Q_j(x)$, and we shall call the set $\partial(rQ_j(x))$ the \textbf{barrier at $x$}.

The following geometric lemma is proven in \cite{O} for cubes in $\mathbb{R}^3$, but it in fact carries over to cubes in any spatial dimension $\mathbb{R}^n$, whenever $n\geq 3$, since the proof only involves the use of dilations. See \cite{O} for details of the proof.
\begin{lemma}[\cite{O}]
\label{GEOMETRIC}
Let $n\geq 3$. Let $x,y$ be two different points in $\mathbb{R}^n$ and let $Q(y)$ and $Q'(x)$ be cubes of sidelength $2a$ and $2b$, respectively. If $\partial(rQ(y))\cap Q'(x)\not= \emptyset$, then $r\in [r_{Q'}-b/a, r_{Q'}+b/a]$, where $r_{Q'}>0$ is such that $x \in \partial(r_{Q'}Q)$.
\end{lemma}
We now proceed to construct a collection of cubes at level $j\geq 0$ that satisfies the barrier property and covers the set $M_j$.
\begin{lemma}
\label{BARRIER}
There exists a collection $\mathcal{B}_j$ of cubes at level $j\geq 0$ that satisfies the barrier property and covers the set $M_j$. In addition, the cardinality of $\mathcal{B}_j$ is at most
\begin{equation}\label{RIGHTCARD}N(\mathcal{B}_j) \leq K 2^{j(n+2-4\alpha+\epsilon)}.\end{equation}
\end{lemma}
\begin{proof}
First we find a cover $\mathcal{B}_j'$ of $M_j$ by cubes at level $j$ so that for any cube $Q$ disjoint from every element of $\mathcal{B}_j'$, there exists a number $r>0$ such that $\partial(rQ)$ is disjoint from every bad cube at level $k$ for all $k\geq j$. In addition the cover $\mathcal{B}_j'$ will satisfy
$$N(\mathcal{B}_j') \leq K 2^{j(n+2-4\alpha+\epsilon)}.$$
Then, the cover $\mathcal{B}_j$ will be obtained by replacing each cube $Q$ in the cover $\mathcal{B}_j'$ by its dilation $3Q$ and then covering each dilated cube by at most $10^n$ cubes at level $j$. This process will result in a cover $\mathcal{B}_j$ having the desired cardinality bound, as in Equation (\ref{RIGHTCARD}), but with a slightly larger constant $K$. In addition, for any point $x\not\in \mathcal{B}_j$, $Q_j(x)$ is a cube disjoint from all the elements $\mathcal{B}_j'$, so the barrier property also gets transferred to the cover $\mathcal{B}_j$. We thus proceed to construct the cover $\mathcal{B}_j'$. 

Following \cite{O}, we call a cube $Q$ at level $j$ \textbf{$k$-naughty} if and only if $Q$ intersects more than $\eta 2^{(k-j)(n+2-4\alpha+2\epsilon)}$ elements of $\mathcal{C}_k$, where $\eta$ is a parameter to be chosen later in this proof. We remark here that, for sufficiently large $k$, there are in fact no $k$-naughty cubes because there are only at most $K2^{k(n+2-4\alpha+\epsilon)}$ cubes in the entire collection $\mathcal{C}_k$. We call a cube $Q$ at level $j$ \textbf{naughty} if and only if it is $k$-naughty for all $k\geq j$. As before, we remark that this imposes only finitely many conditions on the cube $Q$. 

We construct a cover $\mathcal{B}_{j,k}$ of all the $k$-naughty cubes at level $j$ such that
$$N(\mathcal{B}_{j,k}) \leq K \eta^{-1} 2^{j(n+2-4\alpha+\epsilon)}2^{\epsilon(j-k)}.$$
Let $Q^1$ be a $k$-naughty cube at level $j$ and given the $k$-naughty cubes $Q^1,\ldots Q^{l}$, let $Q^{l+1}$ be a $k$-naughty cube at level $j$ disjoint from $3Q^i$ for all $i\in \{1,\ldots, l\}$. Then $Q^{l+1}$ must intersect at least $\eta 2^{(k-j)(n+2-4\alpha+2\epsilon)}$ elements of $\mathcal{C}_k$ that are disjoint from the elements of $\mathcal{C}_k$ intersected by the cubes $Q^i$ for $i\in \{1,\ldots, l\}$. Since we have a definite bound on the number of cubes in $\mathcal{C}_k$, we see that this iterative process terminates after 
$$L\leq N(\mathcal{C}_k)/(\eta 2^{(k-j)(n+2-4\alpha+2\epsilon)}) \leq K \eta^{-1} 2^{(n+2-4\alpha)j} 2^{\epsilon(j-k)}$$
steps. In particular the family $\{ 3 Q^1, \ldots 3 Q^L\}$ covers the $k$-naughty cubes at level $j$. Covering each dilated cube $3Q^i$ by at most $4^n$ cubes at level $j$ gives us our desired cover. 

Next we let 
$$\mathcal{B}_j':= \cup_{k\geq j} \mathcal{B}_{j,k},$$
which covers the set $M_j$, because $\mathcal{B}_{j,j}$ covers all the $j$-naughty cubes at level $j$. Let $Q$ be a cube disjoint from all elements of $\mathcal{B}_j'$. Recall that we want to find some $0<r<1$ such that $\partial(rQ)$ does not intersect any bad cube at level $k$ for any $k\geq j$. Let $C^k(Q)$ be the collection of cubes $Q'$ from $\mathcal{C}_k$ intersecting $Q$. Since $Q$ is not naughty, we have
$$N(C^k(Q)) \leq \eta 2^{(k-j)(n+2-4\alpha+2\epsilon)}.$$
Now let $Q'$ be an arbitrary element of $\mathcal{C}_k$ that intersects $\partial(rQ)$ with $2^{-(1-\epsilon)}<r<1$. Since $r<1$, we necessarily have that $Q'\in C^k(Q)$, and since $r> 2^{-(1-\epsilon)}$, we necessarily have that the center of $Q'$ is different than the center of $Q$ (this is a requirement on $r$ that is present but unexplained in \cite{KP} while omitted in \cite{O}). Now we may let $r_{Q'}$ be such that $\partial(r_{Q'}Q)$ contains the center of $Q'$, so by the geometric Lemma \ref{GEOMETRIC} (with $2a = 2^{-j(1-\epsilon)}$ and $2b = 2^{-k(1-\epsilon)}$), we have
$$\partial(rQ) \cap Q' \not= \emptyset \Longrightarrow r\in [r_{Q'} - 2^{(1-\epsilon)(j-k)}, r_{Q'} + 2^{(1-\epsilon)(j-k)}].$$
Thus if $f_k(r)$ is the number of bad cubes at level $k$ that intersect $\partial(rQ)$, we may estimate
$$f_k(r) \leq \sum_{Q'\in C^k(Q)} \chi_{[r_{Q'} - 2^{(1-\epsilon)(j-k)}, r_{Q'} + 2^{(1-\epsilon)(j-k)}]}(r).$$
Integrating we have $$\|f_k(r)\|_{L^1(2^{-(1-\epsilon)},1)} \leq 2 N(C^k(Q)) 2^{(1-\epsilon)(j-k)}\leq 2 \eta 2^{(4\alpha-n-1-3\epsilon)(j-k)}.$$
Denoting $\sum_{k\geq j} f_k := f$, we conclude, by using that $\alpha>(n+1)/4$ (choosing $0<\epsilon < (4\alpha-n-1)/3$) and summing, that
$$\|f\|_{L^1((2^{-(1-\epsilon)},1))} \leq c\eta.$$
By choosing $\eta$ small enough we can use the fact that $f$ takes only integer values to guarantee the existence of $r\in (2^{-(1-\epsilon)},1)$ such that $f(r)=0$. This number $r$ is the number desired so that $\mathcal{B}_j'$ has the barrier property. Indeed, by Chebyshev's inequality and assuming $\eta$ sufficiently small (here $m(\cdot)$ denotes the Lebesgue measure on $\mathbb{R}$)
$$m(\{r\in (2^{-(1-\epsilon)}, 1) : f(r)\geq 1\}) \leq \int_{2^{-(1-\epsilon)}}^1 \abs{f} dr < \frac{1}{1000}(1-2^{-(1-\epsilon)}),$$
so there certainly exists some $r\in (2^{-(1-\epsilon)},1)$, with the desired property.
\end{proof}

We let 
$$E := \limsup_{j\to \infty} \mathcal{B}_j$$
and our goal is to show 
\begin{equation}
\label{GOAL22}
S_T \subset E.
\end{equation}
By the bound in Equation (\ref{RIGHTCARD}) and Lemma \ref{dimcompute}, we know that $\mathcal{H}(E)\leq n+2-4\alpha+\epsilon$, which combined with Equation (\ref{GOAL22}) and the arbitrary smallness of $\epsilon$, will finish the proof of Theorem \ref{PRMAINMAIN}.

\subsection{Regularity Away from Naughty Cubes}

We prove Equation (\ref{GOAL22}).
Our partial regularity result requires the use of the barrier in an essential way. The barrier construction allows the critical regularity proven in Proposition \ref{PRCRITREG} to propagate throughout an open neighborhood of every point $x\not\in E$, which is essential for showing regularity within an open set (in the sense of $u(x,t)$ having two continuous spatial derivatives). 

\begin{theorem}
Let $x\not\in E$. There exists $j_1(x)\in \mathbb{Z}$ such that 
$$u_Q(t)<  2^{-j \rho(Q)/2}$$
for all times $t\in [0,T]$ and for every cube $Q\subset r Q_{j_1}(x)$ at any level $j$. Here $r:=r(x,j_1)\in (2^{-(1-\epsilon)},1)$ is the number found in the proof of Lemma \ref{BARRIER}, 
$$\rho(Q) := n+2-4\alpha + \min(2n+7, \epsilon \delta(Q)/10 ),$$
and $\delta(Q)$ denotes the smallest integer $k$ such that $Q_{j-k}$ intersects $\partial(r Q_{j_1}(x))$.
\end{theorem}
\begin{proof}
First we observe that $x\not \in E$ implies that there exists some integer $j_2\geq 0$ such that $x\not \in Q$ for any $Q\in \mathcal{B}_j$ and for any $j\geq j_2$. Recall that $j_0$ is the integer found in the proof of Proposition \ref{PRCRITREG}. Suppose the claim of the Theorem is not true, then for any $j_1\geq (\max(j_0, j_2)+10)/\delta^2$ (we assume that $j_1$ is much larger than $j_0$ in order to apply Proposition \ref{PRCRITREG} on dilations of cubes) we can let $t_0>0$ be the first time such that
\begin{equation}\label{CONTRA}u_{Q'}(t)\leq 2^{-k\rho(Q')/2}\end{equation}
for all $t\in [0,t_0]$ and all cubes $Q'\subset rQ_{j_1}(x)$ at any level $k$ while there exists some $j$ and some cube $Q\subset rQ_{j_1}$ at level $j$ such that 
\begin{equation}
\label{CONTRA2}
u_Q(t_0)> 2^{-j\rho(Q)/2}.
\end{equation}
That Equation (\ref{CONTRA2}) holds follows from continuity of the energy integral proven in Lemma \ref{Integral1}, also compare with Equation (3.32) in \cite{O}. Once again, let $t_1$ be the largest time such that $u_Q(t_1)= \frac{1}{2}2^{-j\rho(Q)/2}$, so that
\begin{equation}
\label{CONTRA3}\frac{1}{2} 2^{-j\rho(Q)/2} \leq u_Q(t) \leq 2^{-j\rho(Q)/2}\end{equation}
for all $t\in [t_1,t_0]$. 
Recall that our estimate in Equation (\ref{BIGEST2}) states that for any $\delta>0$:
$$\abs{ \langle-\hat{B}(u,u),P_j \phi_{Q,j}^2P_j u\rangle} \leq K2^{j(1+\frac{n\delta}{2}+\frac{n\gamma}{3})}u_{\mathcal{N}^1(Q)}u_{Q}+$$ $$+K\sum_{k=\delta j}^{j-10} 2^{\frac{nk}{2}+j +\frac{n\gamma j}{3}}u_{Q_k}u_{\mathcal{N}^1(Q)}u_{Q'}+ K \sum_{k>j+10} 2^{nj(\frac{1}{2}-\frac{2\gamma}{3}) +k(n\gamma +1)}u_{Q'}\|\phi_{Q_j,j}P_k u\|_2^2 +$$ $$+K2^{\frac{(n+2)j}{2}+\frac{n\gamma j}{3}}u_{Q}\|\phi_{Q_j,j}P_{j-10\leq k \leq j+10}u\|^2_{L^2}+K2^{-150j}.$$
We thus conclude from Equation (\ref{COEFFODE}) and Equation (\ref{BIGEST1}) that
$$\frac{3}{4} 2^{-j\rho(Q)} =  u_Q^2(t_0)- u_Q^2(t_1)  \leq \int_{t_1}^{t_0}\bigg( K2^{j(1+\frac{n\delta}{2}+\frac{n\gamma}{3})}u_{\mathcal{N}^1(Q)}u_{Q}+$$ $$+K\sum_{k=\delta j}^{j-10} 2^{\frac{nk}{2}+j +\frac{n\gamma j}{3}}u_{Q_k}u_{\mathcal{N}^1(Q)}u_{Q}+ K \sum_{k>j+10} 2^{nj(\frac{1}{2}-\frac{2\gamma}{3}) +k(n\gamma +1)}u_{Q}\|\phi_{Q_j,j}P_k u\|_2^2 +$$ $$+K2^{\frac{(n+2)j}{2}+\frac{n\gamma j}{3}}u_{Q}\|\phi_{Q_j,j}P_{j-10\leq k \leq j+10}u\|^2_{L^2}\bigg)dt- \int_{t_1}^{t_0}K2^{2\alpha j} u_Q^2dt +$$ \begin{equation}\label{BIGEST200}+\int_{t_1}^{t_0} K 2^{(2\alpha-\epsilon)j}u_{\mathcal{N}^1(Q)}^2 +K2^{-150j} .\end{equation}

We first claim that $\delta(Q)\geq 11$. Otherwise if $\delta(Q)\leq 10$ (which implies $\rho(Q)\leq  n+2-4\alpha+\epsilon$), then there exists some $k_0\leq 10$ such that the cube $Q_{j-k_0}$ intersects the barrier $\partial(rQ_{j_1}(x))$. It follows then that the larger cubes $Q_{k}$ also intersect the barrier for all $k\in [\delta j, j-10]$. By Proposition \ref{PRCRITREG}, and since $\delta j > j_0$, we conclude that 
$$u_Q(t_0) < 2^{-j(n+2-4\alpha+\epsilon)/2}\leq 2^{-j\rho(Q)/2},$$
which contradicts Equation (\ref{CONTRA2}). We take note of one consequence of the fact that $\delta(Q)>10$ separately in 
\begin{equation}
\label{FAR}
\rho(Q) \geq n+2-4\alpha +\epsilon.
\end{equation}
Having shown that the cube $Q$ is necessarily "far" from the barrier, we proceed to show a contradiction to our contradiction hypothesis.
First we show that large dilations of $Q$ have critical regularity, or that:
\begin{equation}\label{LARGEQ}u_{Q_k}(t)< K2^{-(n+2-4\alpha+\epsilon)k/2}\end{equation}
for all times $t<t_0$ and for all $k\in [\delta j, j-5]$. If $\delta(Q_k)\geq 11$, then $Q_k \subset rQ_{j_1}(x)$ and $\rho(Q_k)\geq n+2-4\alpha+\epsilon$. Indeed, in this case, the smallest dilation of $Q_k$ that still intersects the barrier is at worst $Q_{k-11}$, so the original cube $Q_k$ remains within the cube $rQ_{j_1}$. But then by Equation (\ref{CONTRA}) we conclude, when $\delta(Q_k)\geq 11$:
$$u_{Q_k}(t)< 2^{-k\rho(Q_k)/2}< K2^{-(n+2-4\alpha+\epsilon)k/2}.$$
In the case when $\delta(Q_k)\leq 10$, we argue as before and get that the larger cubes $Q_{l-10}$ intersect the barrier for all $l\in [\delta k, k]$, so by Proposition \ref{PRCRITREG}, and since $k\geq \delta j > j_0$, we obtain the desired critical regularity.

Next we show that cut-offs at high frequency have better than critical regularity. Let $k\geq j+10(2n+7)/\epsilon$ be arbitrary, we aim to show that 
\begin{equation}\label{HIGHFREQEST}\|\phi_{Q_j,j}P_k u\|_2< K2^{-nj} 2^{-(n+9-4\alpha)k/2}\end{equation}
for all times $t<t_0$. Recall that for $k\geq j$, $Q_k := (1+2^{-\epsilon k})Q$, so, by also choosing $j_1$ sufficiently large (i.e. $j_1>1/\epsilon$), we have that $Q_j \subset (3/2)Q$. We may thus cover $(3/2)Q$ by at most $K 2^{n(k-j)(1-\epsilon)}$ cubes at level $k$ (we denote the cover by $S_k(3Q/2)$). Therefore we estimate:
$$\|\phi_{Q_j,j}P_k u\|_2 \leq K \sum_{Q'\in S_k(3Q/2)} u_{Q'}.$$
Each $Q'\in S_k(3Q/2)$ is a cube at level $k\geq j+10(2n+7)/\epsilon$, so $Q'_j = 2^{-(j-k)(1-\epsilon)}Q'\subset Q_{j-2}$, and in particular $\delta(Q') = \delta(Q_j')+k-j\geq \delta(Q_{j-2}) +10(2n+7)/\epsilon>10(2n+7)/\epsilon$. It follows from the definition that $\rho(Q') \geq 3n+9-4\alpha$. Then we conclude:
$$\|\phi_{Q_j,j}P_k u\|_2\leq N(S_k(3Q/2))2^{-(3n+9-4\alpha)k/2}\leq K 2^{-nj} 2^{\epsilon n(j-k)} 2^{-(n+9-4\alpha)k/2}\leq 2^{-nj}2^{-(n+9-4\alpha)k/2} .$$
For the local frequencies, we would like to show for all $t<t_0$ that for all $k\in (j-5, j+10(2n+7)/\epsilon]$, we have
\begin{equation}\label{LOCALFREQEST}\|\phi_{Q_j,j}P_ku\|_{L^2}\leq K 2^{-j(\rho(Q)-\frac{2\epsilon}{5})/2}.\end{equation}
Let $j_1$ be sufficiently large so that $Q_j \subset (3/2)Q$. Let $k\in (j-5, j+10(2n+7)/\epsilon]$, and let $S_k((3/2)Q)$ be a cover of $(3/2)Q$ by cubes $Q'$ at level $k$. There are at most $K$ such cubes $Q'$, where $K$ depends only on $\epsilon$ and $n$. Moreover, by the fact that $\delta(Q)\geq 11$, we have
$$\delta(Q_k) \geq \delta(Q) +k-j \geq \delta(Q)-4 \geq 7,$$
so, since $r>2^{-(1-\epsilon)}$, we have $Q_k \subset rQ_{j_1}(x)$, $\rho(Q_k)\geq \rho(Q) -\frac{2\epsilon}{5}$, and Equation (\ref{CONTRA}) applies. Thus, we can estimate, if $k<j$,
$$\|\phi_{Q_j,j}P_ku\|_{L^2} \leq u_{Q_k}\leq 2^{-\rho(Q_k)k/2} \leq 2^{-(\rho(Q)-\frac{2\epsilon}{5})k/2} \leq K2^{-(\rho(Q)-\frac{2\epsilon}{5})j/2}.$$
If $k\geq j$, then $Q'_j \subset Q_{j-2}$, so $\delta(Q') = \delta(Q'_j) +k-j \geq \delta(Q_{j-2}) = \delta(Q)-2$, from which $\rho(Q')\geq \rho(Q)-\epsilon/5$. By Equation (\ref{CONTRA}) we conclude
$$u_{Q'} \leq 2^{-\rho(Q')k/2} \leq 2^{-j(\rho(Q)-\epsilon/5)/2}.$$
From which we conclude
$$\|\phi_{Q_j,j}P_ku\|_{L^2}\leq K 2^{-j(\rho(Q)-\epsilon/5)/2},$$
which completes the proof of Equation (\ref{LOCALFREQEST}). Now we begin estimating the terms on the right-hand-side of Equation (\ref{BIGEST200}).
Since $\alpha>(n+1)/4\geq 1$, we can assume that $\delta, \epsilon,\gamma$ are sufficiently small so that
$$1+\frac{n\delta}{2}+\frac{n\gamma}{3}+\epsilon<2<2\alpha.$$
Then, by Equation (\ref{CONTRA3}) and our estimate on the local frequencies in Equation (\ref{LOCALFREQEST}), we have
\begin{equation}\label{STEP200}\int_{t_1}^{t_0} K2^{j(1+\frac{n\delta}{2}+\frac{n\gamma}{3})}u_{\mathcal{N}^1(Q)}u_{Q}dt \leq K(t_0-t_1)2^{(2\alpha-\rho(Q)-4\epsilon/5)j}.\end{equation}
Likewise, by our estimate on the local frequencies in Equation (\ref{LOCALFREQEST}), we have
\begin{equation}\label{STEP201}K\int_{t_1}^{t_0} 2^{(2\alpha-\epsilon)j}u_{\mathcal{N}^1(Q)}^2 dt \leq K(t_0-t_1)2^{j(2\alpha-\rho(Q)-3\epsilon/5))}.\end{equation}
Similarly, by Equation (\ref{CONTRA3}), Equation (\ref{LARGEQ}), and Equation (\ref{LOCALFREQEST}), we can estimate the low-high terms:
$$K\sum_{k=\delta j}^{j-10}\int_{t_1}^{t_0} 2^{\frac{nk}{2}+j +\frac{n\gamma j}{3}}u_{Q_k}u_{\mathcal{N}^1(Q)}u_{Q}dt \leq \sum_{k=\delta j}^{j-10} \int_{t_1}^{t_0}2^{\frac{nk}{2}+j+\frac{n\gamma j}{3}}2^{-k(n+2-4\alpha+\epsilon)/2} 2^{j(-\rho(Q)+\frac{\epsilon}{5})}dt\leq $$
$$\leq \sum_{k=\delta j}^{j-10} \int_{t_1}^{t_0}2^{j+\frac{n\gamma j}{3}-\rho(Q)j+\frac{\epsilon j}{5}}  2^{-k(2-4\alpha+\epsilon)/2}dt\leq K(t_0-t_1)2^{j+\frac{n\gamma j}{3}-\rho(Q)j+\frac{\epsilon j}{5}} 2^{-j(2-4\alpha+\epsilon)/2}.$$
Where we have summed over $k$ after using that $\alpha>1$. Simplifying we have
$$K(t_0-t_1)2^{j+\frac{n\gamma j}{3}-\rho(Q)j+\frac{\epsilon j}{5}} 2^{2\alpha j} 2^{-j(2+\epsilon)/2}\leq K(t_0-t_1)2^{(2\alpha-\rho(Q)-\frac{3\epsilon}{10}+\frac{n\gamma}{3})j}.$$
We can choose $\gamma< 3\epsilon/(100n)$ so that
\begin{equation}\label{STEP202}K\sum_{k=\delta j}^{j-10}\int_{t_1}^{t_0} 2^{\frac{nk}{2}+j +\frac{n\gamma j}{3}}u_{Q_k}u_{\mathcal{N}^1(Q)}u_{Q}dt \leq  K(t_0-t_1)2^{j(2\alpha-\rho(Q)-\frac{\epsilon}{10})j} .\end{equation}
We estimate the high-high terms as follows:
$$\int_{t_1}^{t_0}K \sum_{k>j+10} 2^{nj(\frac{1}{2}-\frac{2\gamma}{3}) +k(n\gamma +1)}u_{Q}\|\phi_{Q_j,j}P_k u\|_2^2\leq$$
$$\int_{t_1}^{t_0}K \sum_{k>j+10(2n+7)/\epsilon} 2^{nj(\frac{1}{2}-\frac{2\gamma}{3}) +k(n\gamma +1)}u_{Q}\|\phi_{Q_j,j}P_k u\|_2^2+$$ $$+\int_{t_1}^{t_0}K \sum_{k>j+10}^{j+10(2n+7)/\epsilon} 2^{nj(\frac{1}{2}-\frac{2\gamma}{3}) +k(n\gamma +1)}u_{Q}\|\phi_{Q_j,j}P_k u\|_2^2 \leq$$
 $$\leq K\int_{t_1}^{t_0}2^{-\rho(Q)j/2}2^{nj(1/2-2\gamma/3)}2^{-2nj}\sum_{k>j+10} 2^{k(n\gamma +1-(n+9-4\alpha))}+$$
 $$K\int_{t_1}^{t_0}2^{-\rho(Q)j/2}2^{nj(1/2-2\gamma/3)}2^{j(n\gamma +1-\rho(Q)+\frac{2\epsilon}{5})}.$$
Since $\alpha<(n+2)/4$, we can choose $\gamma$ small enough (e.g. $\gamma<6/n$) so that the infinite sum is convergent, yielding
$$\int_{t_1}^{t_0}K \sum_{k>j+10} 2^{nj(\frac{1}{2}-\frac{2\gamma}{3}) +k(n\gamma +1)}u_{Q}\|\phi_{Q_j,j}P_k u\|_2^2\leq$$ $$\leq K(t_0-t_1)2^{-\rho(Q)j/2}2^{nj(1/2-2\gamma/3)}2^{-2nj}2^{j(n\gamma +1-(n+9-4\alpha))}+$$
 $$K(t_0-t_1)2^{-\rho(Q)j/2}2^{nj(1/2-2\gamma/3)}2^{j(n\gamma +1-\rho(Q)+\frac{2\epsilon}{5})}.$$
 Recall that by Equation (\ref{FAR}), we know $-\rho(Q) \leq -(n+2-4\alpha+\epsilon)$. It follows that (after choosing $\gamma<3\epsilon/(100n)$):
  $$K2^{-\rho(Q)j/2}2^{nj(1/2-2\gamma/3)}2^{j(n\gamma +1-\rho(Q)+\frac{2\epsilon}{5})} \leq K2^{(2\alpha -\rho(Q) -\epsilon/20)j}.$$
Recall that (by definition) $\rho(Q)\leq 3n+9-4\alpha$, so that (choosing $\gamma<3/n$ and $\epsilon<n+5/2$)
$$K2^{-\rho(Q)j/2}2^{nj(1/2-2\gamma/3)}2^{-2nj}2^{j(n\gamma +1-(n+9-4\alpha))} \leq 2^{j(-\rho(Q)+\rho(Q)/2+n/2+n\gamma/3-2n+1-n-(9-4\alpha))j}\leq$$
$$\leq K 2^{j(-\rho(Q)+3n/2+(9-4\alpha)/2-5n/2+n\gamma/3-(8-4\alpha))j}\leq K2^{j(-\rho(Q)-n+2\alpha-7/2+n\gamma/3)}\leq K2^{j(2\alpha-\rho(Q)-\epsilon)}. $$
We can now conclude that:
\begin{equation}\label{STEP203}\int_{t_1}^{t_0}K \sum_{k>j+10} 2^{nj(\frac{1}{2}-\frac{2\gamma}{3}) +k(n\gamma +1)}u_{Q}\|\phi_{Q_j,j}P_k u\|_2^2\leq K(t_0-t_1)2^{j(2\alpha-\epsilon/20-\rho(Q))}.\end{equation}
The estimate for the local terms proceeds as:
$$\int_{t_1}^{t_0} K2^{\frac{(n+2)j}{2}+\frac{n\gamma j}{3}}u_{Q}\|\phi_{Q_j,j}P_{j-10\leq k \leq j+10}u\|_{L^2}^2 dt \leq K\int_{t_1}^{t_0}2^{j(-\rho(Q)-\rho(Q)/2+\frac{n+2}{2}+\frac{n\gamma}{3}+\frac{2\epsilon}{5})}dt\leq$$
$$\leq K\int_{t_1}^{t_0}2^{j(-\rho(Q)-\frac{n}{2}-1+2\alpha-\frac{\epsilon}{2}+\frac{n+2}{2}+\frac{n\gamma}{3}+\frac{2\epsilon}{5})}\leq K(t_0-t_1)2^{j(2\alpha-\rho(Q)-\frac{\epsilon}{10}+\frac{n\gamma}{3})},$$
where we have used Equation (\ref{FAR}). Choosing $\gamma<3\epsilon/(100n)$, we conclude 
\begin{equation}\label{STEP204}\int_{t_1}^{t_0} K2^{\frac{(n+2)j}{2}+\frac{n\gamma j}{3}}u_{Q}\|\phi_{Q_j,j}P_{j-10\leq k \leq j+10}u\|_{L^2}^2 dt \leq K(t_0-t_1)2^{(2\alpha-\rho(Q)-\epsilon/20)j}.\end{equation}
From Equation (\ref{STEP200}), Equation (\ref{STEP201}), Equation (\ref{STEP202}), Equation (\ref{STEP203}), and Equation (\ref{STEP204}), we know, choosing $j_1$ large enough and $\epsilon$ small enough:
$$\frac{3}{4}2^{-j\rho(Q)} \leq K^{-150j}-K\int_{t_1}^{t_0} 2^{2\alpha j}u_Q^2dt + K2^{(2\alpha-\epsilon/20-\rho(Q))j}.$$
Since $u_Q\geq 2^{-\rho(Q)j/2}/2$ on the domain of time-integration, we can choose $j_1$ large enough so that 
$$\frac{3}{4}2^{-j\rho(Q)} \leq -K(t_0-t_1) 2^{(2\alpha -\rho(Q))j}\big(1-K2^{-\epsilon j/20}\big)<0$$
which implies $\frac{3}{4}<0$, a contradiction.
\end{proof}

\begin{corollary}
\label{ACTUALREG}
Let $x\not\in E$. There exists a neighborhood $U$ of $x$ such that $\|u(t)\|_{C^\infty(U)}$ is bounded for all $t\in [0,T]$.
\end{corollary}
\begin{proof}
By changing the definition of $\rho(Q)$, we can get arbitrarily stronger control on the energy away, as long as we are far from the barrier, which would allow us to conclude arbitrarily high regularity in the proof that now follows. It thus suffices to show here that, for example, there exists a neighborhood $U$ of $x$ such that $\|u(t)\|_{C^2(U)}$ is bounded for all $t\in [0,T]$. We have proven that there exists $j_1$ and $r$ such that $u_Q(t)\leq 2^{-j\rho(Q)/2}$ for all $t\in [0,T]$ and for all cubes $Q\subset rQ_{j_1}(x)$ at level $j$. Let $j_2$ be the smallest integer such that $\delta(Q)\geq 10(2n+7)/\epsilon$ for all cubes $Q\subset Q_{j_2}$ at level $j$, so that $\rho(Q)\geq 3n+9-4\alpha$ for any such cube $Q$ at level $j$. Let $U:= Q_{j_2+2}(x)$. 
Let $S_j(U)$ be a cover of $U$ by cubes at level $j$. Then for $j\geq j_2$:
$$\| P_j u\|_{C^2(U)} \leq \|P_j u\|_{W^{3+n/2,2}(U)}\leq 2^{j(3+n/2)}\|P_j u\|_{L^2}\leq  2^{j(3+n/2)}\sum_{Q\in S_j(U)} \|\phi_Q P_j u\|_{L^2} \leq$$ $$\leq K 2^{nj(1-\epsilon)} 2^{nj/2+3j} u_Q\leq 2^{(n(1-\epsilon)+n/2+3-\rho(Q))j/2}\leq K2^{-j/2}.$$
Thus, summing over $j$, we have
$$\|u\|_{C^2(U)} \leq K(j_2).$$

\end{proof}

Corollary \ref{ACTUALREG}, the fact that $\mathcal{H}(E)\leq n+2-4\alpha+\epsilon$, and the arbitrary smallness of $\epsilon$ finishes the proof of Theorem \ref{PRMAINMAIN}.

\section{Estimates for the Local Energy}
\subsection{Estimate for the Dissipation Term}
We first consider the dissipation term in Equation (\ref{COEFFODE}).
\begin{proposition}
\label{dissipationestimate}
Let $Q$ be a cube at level $j\geq 0$. We have:
$$\langle(-\Delta)^\alpha u , P_j \phi_{Q,j}^2 P_j u \rangle \quad \geq -K2^{-200j}+K2^{2\alpha j}u_Q^2 -K 2^{(2\alpha-\epsilon)j}u_{\mathcal{N}^1(Q)}^2$$
\end{proposition}
\begin{proof}
We have:
 \begin{equation}
\label{32EQN}
\langle(-\Delta)^\alpha u , P_j \phi_{Q,j}^2 P_j u \rangle =\langle(-\Delta)^\alpha \phi_{Q,j}P_j u , \phi_{Q,j}P_j u \rangle+\langle[(-\Delta)^\alpha,\phi_{Q,j}P_j] u , \phi_{Q,j} P_j u \rangle.\end{equation}
It remains to estimate the two terms on the right-hand-side of Equation (\ref{32EQN}).
We begin by using Lemma \ref{commutator1} to state:
\begin{equation}\label{300EQN}\abs{\langle(-\Delta)^\alpha \phi_{Q,j}P_j u , \phi_{Q,j}P_j u \rangle -\langle (-\Delta)^\alpha \tilde{P}_j \phi_{Q,j} P_j u, \phi_{Q,j}P_j u \rangle } \leq K 2^{-400j}\end{equation}
and
\begin{equation}\label{3000EQN}\abs{\langle \phi_{Q,j}P_j u , \phi_{Q,j}P_j u \rangle -\langle  \tilde{P}_j \phi_{Q,j} P_j u, \phi_{Q,j}P_j u \rangle } \leq K 2^{-400j}.\end{equation}
Then, because of the localization in frequency space, we have
$$\langle (-\Delta)^\alpha \tilde{P}_j \phi_{Q,j} P_j u, \phi_{Q,j}P_j u \rangle =$$ \begin{equation}\label{400EQN}=\int \abs{\xi}^{2\alpha}\tilde{p}_j(\xi)\abs{\mathcal{F}(\phi_{Q,j}P_j u)(\xi)}^2d\xi \geq K 2^{2\alpha j}\langle \tilde{P}_j \phi_{Q,j}P_j u, \phi_{Q,j}P_j u\rangle.\end{equation}
Using Equation (\ref{300EQN}), Equation (\ref{3000EQN}), and Equation (\ref{400EQN}) we conclude that
\begin{equation}
\label{Step1}\langle(-\Delta)^\alpha \phi_{Q,j}P_j u , \phi_{Q,j}P_j u \rangle \quad \geq -K2^{-200j} + 2^{2\alpha j} u_Q^2.\end{equation}
Now we turn to the second term on the right-hand-side of Equation (\ref{32EQN}). 

First, by Lemma \ref{pseudoproduct}, we have $[(-\Delta)^\alpha,\phi_j P_j] \in OPS^{2\alpha-\epsilon}_{1,1-\epsilon}$. Thus, the symbol of this operator, which we denote by $q(x,\xi)$, satisfies
$$\abs{q(x,\xi)} \leq K((1+\abs{\xi}^2)^{1/2})^{2\alpha-\epsilon}.$$
Because of this bound and the localization of frequency, we get, proceeding as above:
$$\abs{\langle[(-\Delta)^\alpha,\phi_jP_j] u , \phi_{Q,j} P_j u \rangle}  \leq K (2^{(2\alpha-\epsilon)j}+1)\abs{\langle u , \phi_{Q,j} P_j u \rangle}\leq K 2^{(2\alpha-\epsilon)j}\abs{\langle u , \phi_{Q,j} P_j u \rangle}.$$
Letting $\tilde{Q}=(1+2^{-\epsilon j})Q$, we conclude using Lemma \ref{commutator1} (to commute Littlewood-Paley projections), Lemma \ref{bumpcommute} (to commute bump functions), the energy dissipation law, and the Cauchy-Schwartz inequality that:
$$\abs{\langle[(-\Delta)^\alpha,\phi_jP_j] u , \phi_{Q,j} P_j u \rangle} \leq  K 2^{(2\alpha-\epsilon)j}\abs{\langle u,\phi_{\tilde{Q},j}\tilde{P}_j \phi_{Q,j}P_j u \rangle} +K2^{-150j}\leq $$
$$\leq K 2^{(2\alpha-\epsilon)j}\|\phi_{\tilde{Q},j}\tilde{P}_j u\|_2\cdot u_{Q}+K2^{-150j} \leq K 2^{(2\alpha-\epsilon)j}\|\phi_{\tilde{Q},j}\tilde{P}_j u\|_2^2+K2^{-150j} .$$
Finally using Lemma \ref{collectbound}, we have
\begin{equation}
\label{Step2}
\abs{\langle[(-\Delta)^\alpha,\phi_jP_j] u , \phi_{Q,j} P_j u \rangle }\leq  K2^{(2\alpha-\epsilon)j}u_{\mathcal{N}^1(Q)}^2+K2^{-150j}.\end{equation}
Combining Equations (\ref{32EQN}), (\ref{Step1}), and (\ref{Step2}) yields the desired inequality.
\end{proof}

\subsection{Estimate for the Nonlinear Term}
The next term we consider from Equation (\ref{COEFFODE}) is 
$$\langle-\hat{B}(u,u),P_j \phi_{Q,j}^2 P_j u \rangle \quad=\quad\langle-\phi_{Q,j}P_j \hat{B}(u,u),\phi_{Q,j}P_j u \rangle .$$
We shall use a scheme for estimating multilinear functions that is inspired by the Bony-Coifman-Meyer paradifferential calculus (for the details of which, see \cite{BONY}, \cite{COIF}, and \cite{TA}). It involves the following partition of frequency space, which is also used by the authors of \cite{KP} and \cite{O}. We may write
$$P_j\hat{B}(u,u) = H_{j,lh}+H_{j,hl}+H_{j,hh}+H_{j,loc}$$
where
$$\textrm{low-high frequncies:}\quad H_{j,lh} = \sum_{k<j-10} P_j \hat{B}(P_k u, \tilde{P}_j u)$$
$$\textrm{high-low:}\quad H_{j,hl} = \sum_{k<j-10} P_j \hat{B}(\tilde{P}_j u, P_k u)$$
$$\textrm{high-high:}\quad H_{j,hh} = \sum_{k>j+10} P_j \hat{B}(\tilde{P}_k u, P_k u)+\sum_{k>j+10} P_j \hat{B}(P_k u,\tilde{P}_k u)$$
$$\textrm{local:}\quad H_{j,loc} = \sum_{j-10<k<j+10} P_j \hat{B}(\tilde{P}_k u, P_k u)+\sum_{j-10<k<j+10} P_j \hat{B}(P_k u, \tilde{P}_k u).$$
For the rest of this subsection we fix a cube $Q$ at level $j$ and introduce the following notation:
$$k<j \Rightarrow Q_k := 2^{(j-k)(1-\epsilon)}Q,\quad k\geq j\Rightarrow Q_k = (1+2^{-\epsilon k})Q.$$ 
We first estimate the low-high frequency terms:
\begin{lemma}
\label{lowhigh}
There exists a constant $K>0$ such that for any arbitrary $\delta>0$ and $j\in\mathbb{Z}^+$ we have:
$$\abs{\langle-\phi_{Q,j}H_{j,lh},\phi_{Q,j}P_j u  \rangle } \leq$$ $$K2^{j(1+\frac{n\delta}{2}+\frac{n\gamma}{3})}u_{\mathcal{N}^1(Q)}u_Q+K\sum_{k=\delta j}^{j-10} 2^{\frac{nk}{2}+j +\frac{n\gamma j}{3}}u_{Q_k}u_{\mathcal{N}^1(Q)}u_Q+K2^{-150j}.$$
\end{lemma}
\begin{proof}
We first deal with the case when $\delta j \leq k <j-10$ and consider a term in the sum $H_{j,lh}$:
$$\langle-\phi_{Q,j}P_j \hat{B}(P_k u, \tilde{P}_j u), \phi_{Q,j}P_j u \rangle .$$
Now by an application of Lemma \ref{commutator1}, we know that for any $f\in L^2$:
\begin{equation}\label{firstthing}\|\phi_{Q,j}P_jf-\tilde{P}_j\phi_{Q,j}P_jf\|_2 \leq K2^{-200j}.\end{equation}
Since $B$ is an amenable bilinear operator, we have by assumption that $B^1_v(u)=\hat{B}(u,v)$ and $B^2_v(u)=\hat{B}(v,u)$ are pseudodifferential operators with symbols in some class $S^m_{1,1}$. Thus by an application of Lemma \ref{quantbound}, also using Equation (\ref{firstthing}), we see:
\begin{equation} \label{secondthing} \|\phi_{Q_j,j}P_j\hat{B}(P_ku,\tilde{P}_j u)-\phi_{Q_j,j}P_j\hat{B}(\phi_{Q_k,k}P_k u,\tilde{P}_ju)\|_2\leq K2^{-200j}.\end{equation}
We may therefore write
$$\langle-\phi_{Q,j}P_j \hat{B}(P_k u, \tilde{P}_j u), \phi_{Q,j}P_j u \rangle  \quad =\quad \langle\phi_{Q,j}P_j\hat{B}(\phi_{Q_k,k}P_k u, \tilde{P}_j u),\phi_{Q,j}P_j u \rangle +R_{j,k}$$
where $R_{j,k}$ is an error term with absolute value less than $K2^{-195j}$ for all $k$. In a similar fashion, we can commute bump functions to get:
$$\abs{\langle-\phi_{Q,j}P_j \hat{B}(P_k u,\tilde{P}_j u), \phi_{Q,j}P_j u \rangle } \leq$$
\begin{equation}
\label{est0} \abs{\langle\phi_{Q,j}P_j\hat{B}(\phi_{Q_k,k}P_k u, \phi_{Q_j,j}\tilde{P}_j u),\phi_{Q,j}P_j u \rangle }+K2^{-180j}.\end{equation}
We therefore have the estimate:
$$\abs{\langle\phi_{Q,j}P_j\hat{B}(\phi_{Q_k,k}P_k u, \phi_{Q_j,j}\tilde{P}_j u),\phi_{Q,j}P_j u \rangle }\leq K \|\hat{B}(\phi_{Q_k,k}P_k u, \phi_{Q_j,j}\tilde{P}_j u)\|_2 \cdot u_Q.$$
Since $B$ is an amenable bilinear operator, we use the property in Equation (\ref{amenineq}) and get
$$\abs{\langle\phi_{Q,j}P_j\hat{B}(\phi_{Q_k,k}P_k u, \phi_{Q_j,j}\tilde{P}_j u),\phi_{Q,j}P_j u \rangle }\leq$$ \begin{equation}\label{est} u_Q\cdot \|\phi_{Q_k,k}P_k u\|_{\infty}(\|\nabla\phi_{Q_j,j}\tilde{P}_j u\|_{q(\gamma)}+\|\phi_{Q_j,j}\tilde{P}_j u\|_{q(\gamma)}),\end{equation}
where $q(\gamma):= 6/(3-2\gamma)$. For the rest of this proof we refer to $q(\gamma)$ by $q$ only. Now, we use Lemma \ref{ineq2} to get:
\begin{equation}\label{est1}\|\phi_{Q_k,k}P_k u\|_\infty \leq K 2^{nk/2}u_{Q_k} + K2^{-(250/\delta)k}\leq K 2^{nk/2}u_{Q_k} + K2^{-250j}\end{equation}
\begin{equation}\label{est2}\|\phi_{Q_j,j}\tilde{P}_j u\|_q \leq K2^{nj(1/2-1/q)}u_{\mathcal{N}^1(Q)}+K2^{-250j}=K2^{\frac{n\gamma j}{3}}u_{\mathcal{N}^1(Q)}+K2^{-250j}.\end{equation}
We recall that derivatives of a Sobolev-smoothing operator remain Sobolev smoothing, so we may use the "commutator" lemmas of Section 5 "within derivatives" as well. Now, to deal with the gradient term, we use $\tilde{P}_jP_j=P_j$ and Lemma \ref{commutator1} to commute Littlewood-Paley projections and get:
$$\nabla(\phi_{Q_j,j}\tilde{P}_j u ) = \nabla(P_j \phi_{Q_j,j}\tilde{P}_j u ) +R_j = P_j \nabla (\phi_{Q_j, j} \tilde{P}_j u) +R_j,$$ 
where $R_j$ is an error term of absolute value at most $K2^{-300j}$. By Lemma \ref{finiteband1}, we thus have
\begin{equation}\label{est3}\|\nabla(\phi_{Q,j}\tilde{P}_j u )\|_q \leq 2^{j+\frac{n\gamma j}{3}}u_{\mathcal{N}^1(Q)}+K2^{-200j}.\end{equation}

Combining our estimates in Equations (\ref{est0}), (\ref{est}), (\ref{est1}), (\ref{est2}), (\ref{est3}), we get:
$$\abs{\langle-\phi_{Q,j}P_j \hat{B}(P_k u,\tilde{P}_j u), \phi_{Q,j}P_j u \rangle } \leq 2^{nk/2+j+\frac{n\gamma j}{3}}u_{Q_k}u_{\mathcal{N}^1(Q)}u_Q +K2^{-150j}.$$
Summing over $\delta j\leq k<j-10$ and recognizing that there are only around $j$ terms gets us the first term with the sum.
For the case when $k<\delta j$, we directly use the inequality satisfied by $\hat{B}$, our commutator estimates from Section 5, and the energy dissipation law to get:
$$\abs{\langle-\phi_{Q,j}P_j \hat{B}(P_k u,\tilde{P}_j u), \phi_{Q,j}P_j u \rangle } \leq $$
$$\leq2^j u_Q \|P_k u\|_\infty \|\phi_{Q_j,j}\tilde{P}_ju\|_q+K2^{-150j}\leq K2^{j+n\delta j/2+\frac{n\gamma j}{3}}u_Q u_{\mathcal{N}^1(Q)}+K2^{-150j},$$
which finishes the proof.
\end{proof}
By symmetry of $\hat{B}$, the same result holds for the high-low terms:
\begin{lemma}
\label{highlow}
There exists a constant $K>0$ such that for any arbitrary $\delta>0$ and $j\in\mathbb{Z}^+$ we have:
$$\abs{\langle-\phi_{Q,j}H_{j,hl},\phi_{Q,j}P_j u  \rangle } \leq$$ $$ K2^{j(1+\frac{n\delta}{2}+\frac{n\gamma}{3})}u_{\mathcal{N}^1(Q)}u_Q+K\sum_{k=\delta j}^{j-10} 2^{\frac{nk}{2}+j +\frac{n\gamma j}{3}}u_{Q_k}u_{\mathcal{N}^1(Q)}u_Q+K2^{-150j}.$$
\end{lemma}
Our next effort is towards estimating the high-high frequency terms. 

\begin{lemma}
\label{highhigh}
For some constant $K$ and for all $j\in \mathbb{Z}^+$ we have
$$\abs{\langle \phi_{Q,j}H_{j, hh} , \phi_{Q,j}P_j u\rangle }\leq K \sum_{k>j+10}  2^{nj(\frac{1}{2}-\frac{2\gamma}{3})+\gamma j +k(n\gamma +1)}u_Q\|\phi_{Q_j,j}P_k u\|_2^2 + K2^{-150j}.$$
\end{lemma}
\begin{proof}
As usual, we examine and wish to estimate a single term in the high-high sum where $k>j+10$, which by symmetry we may write:
$$\langle \phi_{Q,j}P_j\hat{B}(\tilde{P}_k u, P_k u), \phi_{Q,j}P_j u\rangle.$$
As in the proof of Lemma \ref{lowhigh}, we can commute bump functions across our operators to get
$$\abs{\langle \phi_{Q,j}P_j\hat{B}(\tilde{P}_k u, P_k u), \phi_{Q,j}P_j u\rangle}\leq$$ \begin{equation}
\label{bound0} K\abs{\langle \phi_{Q,j}P_j\hat{B}(\phi_{Q_j,j}\tilde{P}_k u, \phi_{Q_j,j}P_k u), \phi_{Q,j}P_j u\rangle}+K2^{-200j}.\end{equation}
Since $B$ is amenable we can bound the trilinear expression:
$$
\abs{\langle \phi_{Q,j}P_j\hat{B}(\phi_{Q_j,j}\tilde{P}_k u, \phi_{Q_j,j}P_k u), \phi_{Q,j}P_j u\rangle}\leq$$ $$ (\|\phi_{Q_j,j}\tilde{P}_k u\|_{p_1}+\|\nabla\phi_{Q_j,j}\tilde{P}_k u\|_{p_1})\|\phi_{Q_j,j}P_k u\|_{p_2} \|\phi_{Q,j}P_j u\|_r 
$$
where
$$\frac{1}{p_1}+\frac{1}{p_2}+\frac{\gamma}{3}+\frac{1}{r}=1.$$
Henceforth, we can always assume that $0<\gamma<1/2$. With this assumption, we can pick $p_1,p_2,r$ so that
$$\frac{1}{p_1}=\frac{1}{2},\quad \frac{1}{p_2} = \frac{1}{2}-\gamma,\quad \frac{1}{r}= \frac{2\gamma}{3}$$
By Lemma \ref{ineq2} we know that
\begin{equation}\label{bound1}
\|\phi_{Q,j}P_j u\|_r \leq 2^{nj(\frac{1}{2}-\frac{2\gamma}{3})}u_Q +K2^{-200j}\end{equation}
and
\begin{equation}\label{bound2}
\|\phi_{Q_j,j}P_k u\|_{p_2} \leq 2^{nk\gamma}\|\phi_{Q_j,j}P_k\|_2+K2^{-200j} .\end{equation}
By the same reasoning as in the proof of Lemma \ref{lowhigh}, we get
\begin{equation}\label{bound3}
\|\nabla \phi_{Q_j,j}P_k u\|_2 \leq K2^k \sum_{l=-2}^{l=2} \|\phi_{Q_j,j}P_{k+l}u\|_2
\end{equation}
Combining our estimates from Equations (\ref{bound0}), (\ref{bound1}), (\ref{bound2}), (\ref{bound3}) yields:
$$\abs{\langle \phi_{Q,j}H_{j, hh} , \phi_{Q,j}P_j u\rangle }\leq K \sum_{k>j+10} u_Q 2^{nj(\frac{1}{2}-\frac{2\gamma}{3})+k(n\gamma+1)}\|\phi_{Q_j,j}P_k u\|_2^2 + K2^{-150j}.$$
\end{proof}
Next we deal with the "local" frequencies around level $j$:
\begin{lemma}
\label{local}
For some constant $K$ we have for all $j\in \mathbb{Z}^+$:
$$\abs{\langle \phi_{Q,j}H_{j, loc} , \phi_{Q,j}P_j u\rangle }\leq K2^{\frac{(n+2)j}{2}+\frac{n\gamma j}{3}}u_Q \|\phi_{Q_j,j}P_{j-10\leq k\leq j+10}u\|_{L^2}^2+K2^{-150j}.$$
\end{lemma}
\begin{proof}
We can explicitly write out the term we wish to bound by
$$\langle \phi_{Q,j}H_{j, loc} , \phi_{Q,j}P_j u\rangle = 2\sum_{l=-2}^{l=2}\sum_{j-10<k}^{k<j+10}\langle \phi_{Q,j}P_j\hat{B}(P_{k+l}u,P_ku),\phi_{Q,j}P_j u\rangle.$$
As in the proof of Lemma \ref{lowhigh}, we can use the triangle inequality and commute bump functions to get:
$$
\abs{\langle \phi_{Q,j}H_{j, loc} , \phi_{Q,j}P_j u\rangle} \leq$$ \begin{equation}
\label{bound01} K2^{-200j} + K\sum_{l=-2}^{l=2}\sum_{j-10<k}^{k<j+10}\abs{\langle \phi_{Q,j}P_j\hat{B}(\phi_{Q_j,j}P_{k+l}u,\phi_{Q_j,j}P_ku),\phi_{Q,j}P_j u\rangle}
\end{equation}
For appropriate choices of $k=k_0,l=l_0$ that maximize the expression in the sum in Equation (\ref{bound01}), and since our constants $K$ are allowed to depend on $\epsilon$, we have:
$$
K\sum_{l=-2}^{l=2}\sum_{j-10<k<j+10}\abs{\langle \phi_{Q,j}P_j\hat{B}(\phi_{Q_j,j}P_{k+l}u,\phi_{Q_j,j}P_ku),\phi_{Q,j}P_j u\rangle}\leq $$ \begin{equation}\label{bound05} \leq K\abs{\langle \phi_{Q,j}P_j\hat{B}(\phi_{Q_j,j}P_{k_0+l_0}u,\phi_{Q_j,j}P_{k_0}u),\phi_{Q,j}P_j u\rangle}
\end{equation}
We remark that by Lemma \ref{collectbound}:
\begin{equation}
\label{bound02}
\|\phi_{Q_j,j}P_k u\|_2 \leq \| \phi_{Q_j,j} P_{j-10\leq k \leq j+10}\|_{L^2}
\end{equation}
and by Lemma \ref{ineq2}:
\begin{equation}
\label{bound03}
\|\phi_{Q,j}P_j u\|_\infty \leq 2^{nj/2}u_Q +K2^{-200j}.
\end{equation}
As in the proof of Lemma \ref{lowhigh}, we have, where $q= 6/(3-2\gamma)$:
\begin{equation}
\label{bound04}
\|\nabla(\phi_{Q_j,j}P_{k_0+l_0} u )\|_q \leq 2^{j+\frac{n\gamma j}{3}}\| \phi_{Q_j,j} P_{j-10\leq k \leq j+10}\|_{L^2}+K2^{-200j}.
\end{equation}
Combining our estimates from Equations (\ref{bound01}), (\ref{bound05}), (\ref{bound02}), (\ref{bound03}), (\ref{bound04}), we get:
$$\abs{\langle \phi_{Q,j}H_{j, loc} , \phi_{Q,j}P_j u\rangle} \leq K2^{\frac{(n+2)j}{2}+\frac{n\gamma j}{3}}u_Q \|\phi_{Q_j,j}P_{j-10\leq k\leq j+10}u\|_{L^2}^2+K2^{-150j}.$$
\end{proof}

\section{Partial Regularity and Blowup}
In this section we deal only with functions and vector fields defined on $\mathbb{R}^3$. For this reason, the parameter $n$ does not denote the dimension of space but rather an index of summation.
\subsection{Flexibility of Blowup}
We shall assume that the reader has a familiarity with the paper \cite{T} of Tao. We first give an overview of the construction in \cite{T}. Afterwards, we show the flexibility of the blowup result therein. Indeed, Tao notes himself (in footnote 12 of \cite{T}), that his methods show blowup even when hyperdissipation approaches the critical value of $\alpha=5/4$. In this subsection we give some details showing why this is the case while also modifying the technique to prove the existence of a blowup solution to the pseudodifferential equation associated to some amenable operator. 

As before, we let $\langle, \rangle $ be the $L^2$ pairing for vector fields. Let $1<\lambda<2$. We denote by $H^{10}_{df}(\mathbb{R}^3)$ the space of vector fields on $\mathbb{R}^3$ that are divergence free in the distributional sense and whose first ten weak derivatives are square integrable.

 A basic local cascade operator is a bilinear operator $C: H_{df}^{10}(\mathbb{R}^3) \times H_{df}^{10}(\mathbb{R}^3)\to H_{df}^{10}(\mathbb{R}^3)^\ast$ defined via duality by:
$$\langle C(u,v), w \rangle \hspace{5px} =\sum_{n\in \mathbb{Z}}\lambda^{5n/2} \langle u, \psi_{1,n} \rangle \langle v, \psi_{2,n} \rangle  \langle w, \psi_{3,n} \rangle $$
where for $i\in \{1,2,3\}$,
$$\psi_{i,n}(x) = \lambda^{3n/2}\psi_i(\lambda^n x)$$
and $\psi_i$ is a Schwartz vector field whose Fourier transform is compactly supported within a small spherical shell around the unit sphere in $\mathbb{R}^3$. We also define a local cascade operator to be a finite linear combination of basic local cascade operators. Lastly, we call a basic local cascade operator a zero-momentum basic local cascade operator if its constituent Schwartz vector fields satisfy
$$\int_{\mathbb{R}^3} \psi_i dx=0.$$
Likewise, we may define a zero-momentum local cascade operator.

In \cite{T}, Tao proves the existence of a local cascade equation that admits a solution blowing up in finite time. In particular, Tao proves the following:

\begin{theorem}[\cite{T}]
Let $1<\lambda<2$ be arbitrary. Then there exists a symmetric local cascade operator $C$ and a Schwartz divergence-free vector field $u_0$ such that the cancellation identity holds
\begin{equation}
\langle C(u,u),u \rangle =0 \textrm{ for all } u \in H^{10}_{df}(\mathbb{R}^3)\\
\end{equation}
and there does not exist any global mild solution $u: [0,\infty) \to H^{10}_{df}(\mathbb{R}^3)$ to the initial value problem
\begin{equation}
\begin{split}
\partial_t u -\Delta u +C(u,u)=0\\
u(\cdot,t=0) = u_0
\end{split}
\end{equation}
\end{theorem}

In this subsection we show that Tao's result is flexible in the sense that we find a local cascade operator with more restraints on the constituent Schwartz vector fields $\psi_i$ whose corresponding local cascade equation admits a blowup solution with fractional dissipation. In particular we prove:

\begin{theorem}
\label{BLOWUP}
Let $1<\lambda<2$ and $0<\alpha<5/4$ be arbitrary. Then there exists a zero-momentum symmetric local cascade operator $C$ and a Schwartz divergence-free vector field $u_0$ such that the cancellation identity holds
\begin{equation}
\langle C(u,u),u \rangle =0 \textrm{ for all } u \in H^{10}_{df}(\mathbb{R}^3)\\
\end{equation}
and there does not exist any global mild solution $u: [0,\infty) \to H^{10}_{df}(\mathbb{R}^3)$ to the initial value problem
\begin{equation}
\begin{split}
\partial_t u +(-\Delta)^\alpha u +C(u,u)=0\\
u(\cdot,t=0) = u_0
\end{split}
\end{equation}
\end{theorem}

First, we determine the structure our local cascade operator must have in order to fulfill some basic requirements like the cancellation property. In particular, following Tao, we consider four balls $B_1, B_2, B_3, B_4$ in the spherical shell $\{ \xi : 1<\abs{\xi}\leq \frac{\lambda+1}{2}\}$, and we choose the balls so that the unions $B_i\cup(-B_i)$ are mutually disjoint for $i\in \{1,2,3,4\}$. We choose four zero-momentum divergence-free Schwartz vector fields $\psi_i$ such that $\mathcal{F}(\psi_i)$ has support in $B_i \cup -B_i$, and we normalize them so that $\| \psi_i\|_2 =1$. We have the $L^2$-rescaled functions $\psi_{i,n}:= \lambda^{3n/2}\psi_i(\lambda^n x)$ that also satisfy $\| \psi_{i,n}\|_2=1$. Let $S=\{(0,0,0),(1,0,0),(0,1,0),(0,0,1)\}$. We define the following local cascade operator (in other words, a linear combination of basic local cascade operators):
$$C(u,v):= \sum_{n\in \mathbb{Z}} \sum_{\substack{\hspace{5px}i_1,i_2,i_3\in \{1,2,3,4\} \\(\mu_1,\mu_2,\mu_3)\in S} }a_{i_1,i_2,i_3,\mu_1,\mu_2,\mu_3}\lambda^{5n/2}\langle u,\psi_{i_1,n+\mu_1} \rangle \langle v,\psi_{i_2,n+\mu_2} \rangle \psi_{i_3,n+\mu_3}$$
To ensure that $C$ is a symmetric bilinear operator we require that
$$a_{i_1,i_2,i_3,\mu_1,\mu_2,\mu_3}=a_{i_2,i_1,i_3,\mu_2,\mu_1,\mu_3}.$$
The cancellation property is satisfied if we require
$$\sum_{\{a,b,c\}=\{1,2,3\}} a_{i_a,i_b,i_c,\mu_a,\mu_b,\mu_c}=0.$$
With these first conditions imposed on $C$, we are ready to study the basic dynamics of the pseudodifferential equation associated to $C$. 
Whenever applicable, we define the following Fourier projections that will act as "wavelets":
$$u_{i,n}(t)(x) := \mathcal{F}^{-1}(\chi_{\xi\in \lambda^n (B_i \cup -B_i)}(\xi)\cdot \mathcal{F}(u(t))(\xi))$$
and the following functions of time that behave like "wavelet coefficients":
$$X_{i,n}(t):= \langle u(t),\psi_{i,n} \rangle =\langle u_{i,n}(t),\psi_{i,n} \rangle .$$
Lastly, we define the following "local energies":
$$E_{i,n}(t):= \frac{1}{2}\|u_{i,n}(t)\|_2^2.$$
The behavior of the pseudodifferential equation is determined by the dynamics of the functions $X_{i,n}(t)$. This is described in Lemma 4.1 of \cite{T}. Below, we only note the modifications necessary in the case of fractional dissipation (assuming that $\psi$ is a zero-momentum vector field has no effect on the lemma's conclusions).
\begin{lemma}
\label{devep}
Suppose $u$ is a global mild solution to the pseudodifferential equation associated to $C$ with initial data $\psi_{1,n_0}$ for some $n_0$ integer.
For any $i,n$ we have
$$\partial_t X_{i,n} = \sum_{i_1,i_2 \in \{1,2,3,4\}}\sum_{(\mu_1,\mu_2,\mu_3)\in S} a_{i_1,i_2,i,\mu_1,\mu_2,\mu_3}\lambda^{5(n-\mu_3)/2}X_{i_1,n-\mu_3+\mu_1}X_{i_2,n-\mu_3+\mu_2} + O(\lambda^{2\alpha n}E_{i,n}^{1/2})$$
and
$$\frac{1}{2}X_{i,n}^2(t) \leq E_{i,n}(t) \leq \frac{1}{2}X_{i,n}^2(t) + O\bigg(\lambda^{2\alpha n} \int_0^t E_{i,n}(s)ds\bigg).$$
\end{lemma}
\begin{proof}
The vector field $u(t)$ is a mild solution of the pseudodifferential equation associated to $C$. Thus we write
\begin{equation}\label{mild}u(t) = e^{-t(-\Delta)^\alpha}u_0 +\int_0^t e^{(s-t)(-\Delta)^\alpha}C(u(s),u(s))ds.\end{equation}
As noted before, we have the "wavelet" decomposition of $u$:
$$u(t) = \sum_{n,i} u_{i,n}(t)$$
with
$$X_{i,n}(t) = \langle u(t),\psi_{i,n}\rangle = \langle u_{i,n}(t), \psi_{i,n}\rangle.$$
We pair Equation (\ref{mild}) with $\psi_{i,n}$ and get
$$
u_{i,n}(t) = e^{-t(-\Delta)^\alpha}X_{i,n}(0)\psi_{i,n}+  \psi_{i,n}\bigg(\sum_{i_1,i_2 \in \{1,2,3,4\}, (\mu_1,\mu_2,\mu_3)\in S}  a_{i_1,i_2,i,\mu_1,\mu_2,\mu_3}\lambda^{5(n-\mu_3)/2}$$
$$\int_0^t\langle u,\psi_{i_1,n+\mu_1-\mu_3} \rangle \langle u,\psi_{i_2,n+\mu_2-\mu_3} \rangle \bigg).
$$
If we differentiate with respect to the time variable $t$ and simplify we obtain
$$
\partial_t u_{i,n} = -(-\Delta)^\alpha u_{i,n}+$$ \begin{equation}\label{mildcoeff} \sum_{i_1,i_2\in \{1,2,3,4\}, (\mu_1,\mu_2,\mu_3)\in S}  a_{i_1,i_2,i,\mu_1,\mu_2,\mu_3}\lambda^{5(n-\mu_3)/2} X_{i_1,n-\mu_3+\mu_1} X_{i_2,n-\mu_3+\mu_2}\psi_{i,n}
\end{equation}
Now, by the localization in frequency space, we observe that (where $K$ depends only on $\psi_i$):
\begin{equation}\label{IBP}\abs{\langle -(-\Delta)^\alpha u_{i,n}, \psi_{i,n} \rangle} =\abs{ \langle u_{i,n} , -(-\Delta)^\alpha\psi_{i,n}\rangle} \leq K E_{i,n}^{1/2}(t)\lambda^{2\alpha n }.\end{equation}
Pairing Equation (\ref{mildcoeff}) with $\psi_{i,n}$ and using Equation (\ref{IBP}) gets us the first statement in the lemma.
The first inequality in the second statement follows from Cauchy-Schwarz. It remains to prove the second inequality. We recall the "local energy inequality" found by Tao, which is independent of the dissipation parameter $\alpha$. Equation (4.11) in \cite{T} is 
\begin{equation}\label{LEI}
\partial_t E_{i,n} \leq \sum_{i_1,i_2}\sum_{(\mu_1,\mu_2,\mu_3)\in S} a_{i_1,i_2,i,\mu_1,\mu_2,\mu_3} \lambda^{5(n-\mu_3)/2}X_{i_1,n-\mu_3+\mu_1}X_{i_2,n-\mu_3+\mu_1}X_{i,n}.
\end{equation}
Now, if we multiply the evolution inequality for $X_{i,n}$ by $X_{i,n}$, subtract this from Equation (\ref{LEI}), and integrate in time, we have the second desired inequality
$$E_{i,n}(t) \leq \frac{1}{2}X_{i,n}^2(t) +K\lambda^{2\alpha n} \int_0^t E_{i,n}(s)ds.$$
\end{proof}
The only significant change from the work in \cite{T} is the presence of $\alpha$ in the exponent of $\lambda$ in the dissipation term.

Tao's construction of a blowup solution to the pseudodifferential equation associated to $C$ continues with an intricate analysis of an infinite dimensional ODE system whose dynamics are contained within the dynamics of $C$. Tao proves that there exist a sequence of times $t_m$ converging to some finite $T$ so that for some $\epsilon_0>0$, at least $\lambda^{-\epsilon_0 n}$ of the energy is concentrated in the ball of radius $\lambda^{-n}$ around the spatial origin. 

 An essential part of the blowup construction is that the aforementioned concentration of energy occurs on timescales of order $t_{n+1}-t_m= \lambda^{(-\frac{5}{2}+O(\epsilon_0))n}$ which are signicantly smaller than the dissipation time scale which (for fractional dissipation of order $\alpha$) is $\lambda^{2\alpha n}$ by Lemma \ref{devep}. This can occur so long as $\alpha<5/4$. Since Tao never uses any assumption about the integral of the functions $\psi_{i,n}$, one can follow through Tao's construction and get a blowup solution to the pseudodifferential equation (with hyperdissipation of order $\alpha<5/4$) associated to some zero-momentum cascade operator. This is nothing less than Theorem \ref{BLOWUP} above.

\subsection{Additional Properties of Local Cascade Operators}
We have established Theorem \ref{BLOWUP}, where we constructed a zero-momentum, symmetric, local cascade operator $C(u,v)$ satisfying the cancellation identity for divergence-free vector fields whose associated $\alpha$-dissipative pseudodifferential equation has a solution blowing up in finite time from Schwartz initial data.  In Section 2, we proved a partial regularity result for $\alpha$-dissipative pseudodifferential equations with nonlinearities arising from amenable bilinear operators. It follows that if we desire to use Theorem \ref{PRTHM} and get a partial regularity result for the pseudodifferential equation from Theorem \ref{BLOWUP}, we have to verify that $C(u,v)$ is an amenable bilinear operator. 

That $C(u,v)$ is defined for Schwartz vector fields is clear. That $C(u,v)$ satisfies the cancellation identity 
$$\langle\mathbb{P}C(u,u), u \rangle =0\quad \textrm{ for divergence-free } u$$
is equally clear. Indeed, since $C(u,u)$ is always divergence-free, we have $\mathbb{P}C(u,u)=C(u,u)$. It follows that
$$\langle \mathbb{P}C(u,u),u\rangle=\langle C(u,u),u\rangle=0 \textrm{ for all divergence-free } u.$$
The scaling property, with the choice of scaling parameter $\lambda$, is clearly satisfied by $C$. Tao also makes this observation. Thus, it remains to show that $C^1_v(u)=\mathbb{P}C(u,v)$ is a pseudodifferential operator in some class $OPS^m_{1,1}$ and to show the bilinear operator bound. First we need to prove an auxiliary lemma.

\begin{lemma}
\label{momentuse}
Let $\psi$ be a Schwartz divergence-free vector field that has zero momentum, i.e. $\int \psi =0$. Let $\Psi_i$ be the vector field given by Lemma \ref{DivF} so that $\textrm{div } \Psi_i = \psi_i$, the $i^{th}$ component of $\psi$. Let $\psi_{n}:= \lambda^{3n/2}\psi(\lambda^n x)$ be the $L^2$ rescaling of $\psi$ and let $\Psi_{i,n}$ be the $L^2$ rescaling of $\Psi_i$. Then for any vector field $v\in H^{10}_{df}$, we have:
$$\lambda^n \langle \psi_n, v\rangle =-\sum_i \int \Psi_{i,n}\cdot \nabla v_i $$
\end{lemma}
\begin{proof}
Since $\psi$ has zero momentum, we can use Lemma \ref{DivF} and write
$$\lambda^n \psi_{i,n}(x)= \lambda^{5n/2}\psi_i (\lambda^n x) = \lambda^{5n/2}(\textrm{div } \Psi_i)(\lambda^n x)=\lambda^{3n/2} \textrm{div }(\Psi_i(\lambda^n x))= \textrm{div } \Psi_{i,n}(x).$$
Then we integrate by parts:
$$\lambda^n\langle \psi_n, v\rangle = \sum_i \int \lambda^{5n/2}\psi_{i}(\lambda^n x)v_i dx = \sum_i \int v_i \textrm{div } \Psi_{i,n}dx =-\sum_i \int \Psi_{i,n} \cdot \nabla v_i dx$$
\end{proof}
\begin{proposition}
\label{LPBOUNDS}
Let $C(u,v)$ be a zero-momentum, symmetric local cascade operator. Then, for arbitrary $0<\gamma<1$ and for any $1\leq p_1, p_2, r\leq \infty$ satisfying
$$\frac{1}{p_1}+\frac{1}{p_2}+\frac{\gamma}{3} = \frac{1}{r},$$
we have that the following inequality holds
$$\| C(u,v) \|_{L^r} \leq K \| u\|_{L^{p_1}}(\|v\|_{L^{p_2}} +\| \nabla v\|_{L^{p_2}})$$
for all $u\in L^{p_1}$ and $v\in W^{1,p_2}$.
The constant $K$ depends on $\gamma,\lambda$, and the vector fields $\psi_i$ used in the definition of $C(u,v)$.
\end{proposition}
\begin{proof}
It suffices to prove the proposition for any (zero-momentum, symmetric) basic local cascade operator:
$$C(u,v)=\sum_{n\in \mathbb{Z}}\lambda^{5n/2} \langle u, \psi_{1,n} \rangle \langle v, \psi_{2,n} \rangle  \psi_{3,n}(x).$$
Since the sum above is absolutely convergent, we can divide the operator $C$ into two parts:
$$C^+(u,v) := \sum_{n\geq 0}\lambda^{5n/2} \langle u, \psi_{1,n} \rangle \langle v, \psi_{2,n} \rangle  \psi_{3,n}(x)$$
and $C^-(u,v):= C(u,v) - C^+(u,v)$.
By a change of variables we observe that
\begin{equation}\label{57}\| \psi_{3,n} \|_{L^r} = \lambda^{3n/2} \| \psi_3(\lambda^n x)\|_{L^r} = \lambda^{3n/2-3n/r}\| \psi_3(u)\|_{L^r}=\lambda^{3n/2-3n/r}\| \psi_3\|_{L^r}.\end{equation}
By the absolute convergence of the sum, we may bound each component in turn. We begin with $C^-$, which is the more straightforward term. We observe that
$$\| C^-(u,v)\|_{L^r} =\| \sum_{n< 0}\lambda^{5n/2} \langle u, \psi_{1,n} \rangle \langle v, \psi_{2,n} \rangle  \psi_{3,n}(x)\|_{L^r}\leq$$
$$\leq \sum_{n< 0}\abs{\lambda^{5n/2} \langle u, \psi_{1,n} \rangle \langle v, \psi_{2,n} \rangle }\cdot \|\psi_{3,n}(x)\|_{L^r}\leq$$
$$\leq \|\psi_3\|_{L^r}\sum_{n<0} \lambda^{4n-3n/r}\abs{\langle u,\psi_{1,n} \rangle }\abs{\langle v,\psi_{2,n} \rangle }\leq$$
$$\leq \|\psi_3\|_{L^r}\sum_{n<0} \lambda^{4n-3n/r}\|u\|_{L^{p_1}}\|\psi_{1,n}\|_{L^{p_1'}}\|v\|_{L^{p_2}}\|\psi_{2,n}\|_{L^{p_2'}}$$
where the first inequality is the triangle inequality, the second is due to Equation (\ref{57}), and the third inequality is Holder's inequality, where $p'$ denotes the Holder conjugate exponent of a number $p$. Then, by the same reasoning that justifies Equation (\ref{57}), we can perform a change of variables in the integration and get
$$\|C^-(u,v)\|_{L^r} \leq \|\psi_3\|_{L^r}\|\psi_{1}\|_{L^{p_1'}}\|\psi_{2}\|_{L^{p_2'}}\sum_{n<0} \lambda^{7n-3n(1/r+1/p_1'+1/p_2')}\|u\|_{L^{p_1}}\|v\|_{L^{p_2}}.$$
Now by the hypothesis on our exponents, it follows that
$$\frac{1}{r}+\frac{1}{p_1'}+\frac{1}{p_2'}=\frac{1}{r}+2-\frac{1}{p_1}-\frac{1}{p_2}=2+\frac{\gamma}{3}.$$
Thus, we conclude, since $0<\gamma<1$ and $\lambda>1$:
$$\|C^-(u,v)\|_{L^r} \leq \|\psi_3\|_{L^r}\|\psi_{1}\|_{L^{p_1'}}\|\psi_{2}\|_{L^{p_2'}}\sum_{n<0}\lambda^{n(1-\gamma)}\|u\|_{L^{p_1}}\|v\|_{L^{p_2}}\leq K \|u\|_{L^{p_1}}\|v\|_{L^{p_2}}$$
where $K$ depends only on $\lambda, \gamma, \psi_i$. Now we proceed to estimate $C^+$. As with $C^-$ we can use Holder's inequality, a change of variables in the integration, and Lemma \ref{momentuse} (to exchange $\lambda^n$ for a derivative) to get:
$$\| C^+(u,v)\|_{L^r}\leq  \|\psi_3\|_{L^r}\|\psi_{1}\|_{L^{p_1'}}\sup_i \|\Psi_{2,i}\|_{L^{p_2'}}\sum_{n\geq 0} \lambda^{-\gamma n}\|u\|_{L^{p_1}}\|\nabla v\|_{L^{p_2}}$$
where $\textrm{div } \Psi_{2,i} = \psi_{2,i}$. This is where we use the zero-momentum condition.
Because $\gamma>0$ we conclude that
$$\| C^+(u,v)\|_{L^r}\leq  K\|u\|_{L^{p_1}}\|\nabla v\|_{L^{p_2}}$$
where $K$ depends only on $\lambda, \gamma, \psi_i$. Combining our estimates with one final application of the triangle inequality gets us
$$\| C(u,v)\|_{L^r}\leq  K\|u\|_{L^{p_1}}(\|v\|_{L^{p_2}}+\|\nabla v\|_{L^{p_2}})$$
where $K$ depends only on $\lambda, \gamma, \psi_i$. 
\end{proof}
We note that the conclusion of Proposition \ref{LPBOUNDS} is in fact stronger than what is necessary for $C$ to be amenable. 
To finish showing that $C(u,v)$ is an amenable bilinear operator, it remains to prove that $C^1_v(u)=\mathbb{P}C(u,v)$ is a pseudodifferential operator in some class $S^m_{1,1}$. That $C^2_v(u)$ is also in the same symbol class would follow immediately from the symmetry of $C$. 
\begin{proposition}
Let $C(u,v)$ be a symmetric local cascade operator, then $C^1_v(u)$ is a pseudodifferential operator with symbol in the class $S^{5/2}_{1,1}$.
\end{proposition}
\begin{proof}
It suffices to prove the proposition for basic local cascade operators of the form:
$$C(u,v)=\sum_{n\in \mathbb{Z}}\lambda^{5n/2} \langle u, \psi_{1,n} \rangle \langle v, \psi_{2,n} \rangle  \psi_{3,n}(x)$$
where we only deal with scalar functions. Using the fact that the Fourier transform is a unitary operator on $L^2$ and the absolute convergence of the sum, we write
$$C^1_v(u) = \iint \sum_{n\in \mathbb{Z}} \lambda^{5n/2} \psi_{3,n}(x) \langle v,  \psi_{2,n}\rangle\mathcal{F}(\psi_{2,n})(\xi) \mathcal{F}(u)(\xi) d\xi.$$
Expanding this and putting it into the form of a pseudodifferential operator we write (up to some dimensional constant):
$$C^1_v(u) = \iint \bigg(\sum_{n\in \mathbb{Z}} \lambda^{5n/2} \psi_{3}(\lambda^n x) \langle v,  \psi_{2,n}\rangle\mathcal{F}(\psi_{2})(\lambda^{-n}\xi)e^{-i x\cdot \xi} \bigg) \mathcal{F}(u)(\xi)e^{i x\cdot \xi} d\xi.$$
The infinite sum, however, has only one nonzero term for any given choice of $\xi$, because $\mathcal{F}(\psi_2)$ has nicely chosen compact support. In particular, the only nonzero term for any given $\xi$ is the $n^{th}$ term where $\abs{\xi} \sim \lambda^{n}$. It follows readily then that the symbol of the pseudodifferential operator $C^1_v$, the "infinite" sum in the expression above, is in fact in the symbol class $S^{5/2}_{1,1}$. 
\end{proof}
As a consequence of the two previous propositions as well as Theorems \ref{BLOWUP} and \ref{PRTHM}, we have:
\begin{theorem}
\label{MAIN}
Let $C(u,v)$ be the bilinear operator from Theorem \ref{BLOWUP}. Let $T$ be the blowup time to the associated pseudodifferential equation with $1/2<\alpha<5/4$.  Then the closed set $S_T$ has Hausdorff dimension at most $5-4\alpha$.
\end{theorem}
We recall here the following general theorem about local cascade operators proven by Tao in \cite{T}.
\begin{theorem}[Theorem 3.2 in \cite{T}]
Let $\lambda_0>1$ be an absolute constant sufficiently close to $1$. Then every local cascade operator (not necessarily zero-momentum) is an averaged (in the sense of Equation (\ref{AVERAGE})) Euler bilinear operator.
\end{theorem}
Due to the theorem above, we have the following corollary of Theorem \ref{MAIN}  (taking $\lambda$ near enough to $1$):
\begin{corollary}
\label{MAIN2}
Let $1/2<\alpha<5/4$ be arbitrary.
There exists a symmetric averaged Euler bilinear operator obeying the cancellation identity for $u\in H^{10}_{df}(\mathbb{R}^3)$ and a Schwartz divergence-free vector initial vector field so that there is no global-in-time solution to the associated $\alpha$-dissipative pseudodifferential equation Moreover if $T$ is the time of first blowup, then the closed set $S_T$ has Hausdorff dimension at most $5-4\alpha$.
\end{corollary}

\section{Technical Lemmas}
Throughout this subsection we shall work in the general Euclidean space $\mathbb{R}^n$. All our spaces will be of functions with domain $\mathbb{R}^n$, unless otherwise stated.
Our shorthand for the $L^p$ norm will be $\|\cdot \|_p$, and we denote the Fourier transform by $\mathcal{F}$. We denote the fractional (inhomogeneous) Sobolev space of order $s$ and integrability $p$ by:
$$W^{s,p}(\mathbb{R}^n):= \bigg\{f\in L^p(\mathbb{R}^n) : \mathcal{F}^{-1}\big((1+\abs{\xi}^2)^{\frac{s}{2}}\mathcal{F}(f)(\xi)\big)\in L^p(\mathbb{R}^n)\bigg\}.$$
Lastly, we denote $W^{s,2}(\mathbb{R}^n)$ by $H^s(\mathbb{R}^n)$ and use the shorthand $\|\cdot \|_s$ for the norm of $H^s(\mathbb{R}^n)$. Since we almost always deal with the spaces $H^s(\mathbb{R}^n)$, the use of this shorthand will always be clear in context.
The following embedding theorem is standard. A proof may be found in \cite{BCD}.
\begin{lemma}[Theorem 1.66 in \cite{BCD}]
\label{EMBED}
The space $H^s(\mathbb{R}^n)$ embeds continuously in the H\"{o}lder Space $C^{k,\alpha}(\mathbb{R}^n)$ provided that $s\geq n/2 +k+ \alpha$.
\end{lemma}
We shall also require the use of the Schwartz-Paley-Wiener theorem, which we state as the following lemma.
\begin{lemma}\label{SPW}
The vector space $C^\infty_c(\mathbb{R}^n)$ of smooth and compactly supported functions is algebraically and topologically isomorphic, via the Fourier transform, to the space of entire functions $F$ on $\mathbb{C}^n$ which satisfy the following growth bounds: there exists $B>0$ such that for every $M\in \mathbb{Z}^+$ there exists real constant $K_M$ such that for all $\xi\in \mathbb{C}^n$,
$$\abs{F(\xi)} \leq K_M (1+\abs{\xi})^{-M} e^{B\abs{\textrm{Im }(\xi)}}.$$
\end{lemma}
\subsection{Pseudodifferential Operators}
 A pseudodifferential operator is an operator $P$ on Schwartz functions $u\in \mathcal{S}$ with the property that:
$$P(x,D)u := (2\pi)^{-n/2}\int e^{i x\cdot \xi}p(x,\xi)\mathcal{F}(u)(\xi)d\xi.$$
where $p(x,\xi)$ is a smooth function on $\mathbb{R}^{2n}$. It is clear that such operators map the Schwartz space into itself. We can write $Pu=\mathcal{F}^{-1}(p\cdot\mathcal{F}(u))$ if the function $p(x,\xi)$ is a function of $\xi$ only. These operators are often called Fourier multipliers. Operators that are Fourier multipliers commute. In particular, if $P_1, P_2$ are two Fourier multipliers, then $P_1P_2= P_2 P_1$.

We call $p(x,\xi)$ the symbol of $P$, and we say that $p$ is in the symbol class $S^m_{\rho,\delta}$ where $m\in \mathbb{R}, 0\leq \rho,\delta\leq 1$ if and only if 
$$\abs{D^{\beta}_xD^\alpha_\xi p(x,\xi)} \leq C_{\alpha,\beta}((1+\abs{\xi}^2)^{1/2})^{m-\rho\abs{\alpha}+\delta\abs{\beta}}$$
for some constants $C_{\alpha,\beta}$. The associated operator $P$ is said to be in the class $OPS^m_{\rho,\delta}$. If $\delta<1$, pseudodifferential operators in the class $S^m_{\rho,\delta}$ can be defined instead on the dual of the Schwartz space, which we denote $\mathcal{S}'$. 

Let $\Lambda^s$ be the pseudodifferential operator with symbol $(1+\abs{\xi}^2)^{s/2}$, which is in the class $S^s_{1,0}$. We may recognize the Sobolev space $H^s(\mathbb{R}^n)$ as $\Lambda^{-s}L^2(\mathbb{R}^n)$.
We call a pseudodifferential operator, $P$, Sobolev-smoothing if and only if $P(H^s)\subset H^t$ for any real numbers $s$ and $t$. In other words, a Sobolev-smoothing operator $P$ continuously maps any Sobolev space into any other Sobolev space. We also define the symbol class $S^{-\infty}_{\rho,\delta}$ to be
$$S^{-\infty}_{\rho,\delta} := \bigcap_{m\in \mathbb{R}} S^m_{\rho,\delta}$$
However, there is really only one class of symbols of infinitely negative order, a fact which we note below.
\begin{lemma}
Let $0\leq \rho,\delta \leq 1$ and let $p(x,\xi)\in S^{-\infty}_{\rho,\delta}$. Then $p(x,\xi)\in S^{-\infty}_{1,0}$.
\end{lemma}
\begin{proof}
Suppose $p(x,\xi)\in S^{-\infty}_{\rho,\delta}$. Let $\alpha,\beta$ and real number $N$ be arbitrary. By hypothesis we can choose $m<N-(1-\rho)\abs{\alpha}-\delta\abs{\beta}$ so that for all $(x,\xi)$:
$$\abs{D^\beta_x D^\alpha_\xi p(x,\xi)}\leq C_{\alpha,\beta} (1+\abs{\xi}^2)^{(m-\rho\abs{\alpha}+\delta\abs{\beta})/2}\leq C_{\alpha,\beta} (1+\abs{\xi}^2)^{(N-\abs{\alpha})/2}.$$
The arbitrariness of $N$ proves that $p(x,\xi)\in S^{-\infty}_{1,0}$.
\end{proof}
Henceforth we denote $S^{-\infty}_{1,0}$ by $S^{-\infty}$. We now recall three standard lemmas below. The proofs may be found in Chapter 0 of \cite{TA}.
\begin{lemma}
\label{pseudoproduct}
Let $P_j \in OPS^{m_j}_{\rho_j,\delta_j}$. Let $\rho = \min(\rho_1,\rho_2)$ and $\delta = \max(\delta_1,\delta_2)$. Suppose
$$0\leq \delta_2 <\rho_1\leq 1.$$
Then $P_1P_2\in OPS^{m_1+m_2}_{\rho,\delta}$, and we have $P_1P_2=T_N+R_N$ where
$$T_N(x,\xi)= \sum_{\abs{\alpha}\leq N} \frac{i^{\abs{\alpha}}}{\alpha!} D^\alpha_\xi p_1(x,\xi)D^\alpha_x p_2(x,\xi)$$
and $R_N \in OPS^{m_1+m_2-N(\rho-\delta)}_{\rho,\delta}$.
\end{lemma}
\begin{lemma}
\label{pseudocont}
If $P \in OPS^m_{\rho,\delta}$, $\rho>0$, and $m<-n+\rho(n-1)$, then 
$$P : L^p(\mathbb{R}^n) \to L^p(\mathbb{R}^n),\quad \forall \hspace{10px}1\leq p\leq \infty$$
is continuous.
\end{lemma}
\begin{lemma}
\label{L2CONT}
If $P \in OPS^0_{\rho,\delta}$ and $0\leq \delta <\rho  \leq 1$, then 
$$P: L^2(\mathbb{R}^n)\to L^2(\mathbb{R}^n)$$ is continuous.
\end{lemma}
The following lemma, which deals with the "exotic" symbol classes $S^m_{1,1}$, is proven in the appendix of \cite{AT}. It provides an asymptotic expansion for operators in the exotic class.
\begin{lemma} 
\label{exotic1}
Let $0\leq \delta <1$.
Suppose $P\in OPS^m_{1,1}$ and $Q\in OPS^\mu_{1,\delta}$. Then $PQ \in OPS^{m+\mu}_{1,1}$ and we have $PQ=T_N+R_N$ where
$$T_N(x,\xi)= \sum_{\abs{\alpha}\leq N} \frac{i^{\abs{\alpha}}}{\alpha!} D^\alpha_\xi p(x,\xi)D^\alpha_x q(x,\xi)$$
and $R_N \in OPS^{m+\mu-N(1-\delta)}_{1,1}$.
\end{lemma}
Our next lemma gives us an inclusion of the operators with symbols in $S^{-\infty}$ into the Sobolev-smoothing operators.
\begin{lemma}\label{SMOOTHIE}
If $P$ has its symbol in $S^{-\infty}$, then it is a Sobolev-smoothing operator.
\end{lemma}
\begin{proof}
Let $s,t$ be arbitrary real numbers, and suppose that $f\in H^t=\Lambda^{-t}L^2$, i.e. we may write $f= \Lambda^{-t}u$ for some function $u\in L^2$. We have to show that $Pf \in H^s=\Lambda^{-s}L^2$, but this amounts to showing that $\Lambda^{s}P\Lambda^{-t} u \in L^2$.

Since $P\in S^{-\infty}$, we have, for some $m$ small enough, that $P\in S^m_{1,0}$ with $m+s-t<-1$. Then, by Lemma \ref{pseudoproduct}, $\Lambda^{s}P\Lambda^{-t} \in S^{m+s-t}_{1,0}$, which is continuous from $L^2 \to L^2$ by Lemma \ref{pseudocont}. This proves that $P$ is Sobolev-smoothing.
\end{proof}
The three lemmas that follow are also about pseudodifferential operators. The first two say that certain products of operators are guaranteed to be Sobolev-smoothing.
\begin{lemma}
\label{smoothing}
Let $P$ be a pseudodifferential operator with symbol $p(\xi)$ in the class $S^m_{1,0}$. Let $\phi_1, \phi_2$ be two smooth functions with disjoint supports (and assume they are pseudodifferential operators in the class $S^0_{1,\delta}$ with $1>\delta\geq0$). Then, the composition $\phi_1P\phi_2$ is a smoothing pseudodifferential operator.
\end{lemma}
\begin{proof}
We first compute the symbol $p_2$ of $P_2=P\phi_2$. Since $1>\delta$, a standard result (see \cite{TA}) gets us the asymptotic expansion:
$$p_2(x,\xi) \sim \sum_{\alpha\geq 0} \frac{i^{\abs{\alpha}}}{\alpha !} D^\alpha_\xi p(\xi) D^{\alpha}_x \phi_2(x)$$
where the sum and asymptotic expression is interpreted in the sense that the sum over $\abs{\alpha}<N$ differs from $p_2(x,\xi)$ by an element of $S^{m-N(1-\delta)}_{1,\delta}$. Then, a similar asymptotic expansion can be found for the symbol of $Q= \phi_1 P_2=\phi_1 P \phi_2$, which is given by:
$$q(x,\xi) \sim \sum_{\alpha\geq 0} \frac{i^{\abs{\alpha}}}{\alpha !} D^\alpha_\xi \phi_1(x) D^{\alpha}_x p_2(x,\xi).$$
It follows that up to a smoothing operator we have:
$$q(x,\xi) \sim \phi_1(x) p_2(x,\xi).$$
Now, for any arbitrary $N$, we have that
$$q(x,\xi) \sim \phi_1(x) \sum_{N>\alpha\geq 0} \frac{i^{\abs{\alpha}}}{\alpha !} D^\alpha_\xi p(\xi) D^{\alpha}_x \phi_2(x)$$ up to an operator in $S^{m-N(1-\delta)}_{1,\delta}$. Since $\phi_1, \phi_2$ have disjoint supports, the above sum is zero for all $N$. We conclude that $q(x,\xi)\in S^{-\infty}_{1,\delta}$, so $Q=\phi_1 P\phi_2$ is an operator in $OPS^{-\infty}$, which, by Lemma \ref{SMOOTHIE}, means that $Q$ is Sobolev-smoothing.
\end{proof}
We separate the next lemma from the previous one since it deals with the exotic class. The proof, however, is the same as for Lemma \ref{smoothing}.
\begin{lemma}
\label{exotic2}
Let $\phi_1,\phi_2$ be two smooth functions with disjoint supports in the class $S^{0}_{1,\delta}$, with $1>\delta\geq 0$. Let $P$ be a pseudodifferential operator in the class $OPS^m_{1,1}$. Then the composition $\phi_1 P\phi_2$ is a Sobolev-smoothing pseudodifferential operator.
\end{lemma}
\begin{proof}
We first deal with the composition $P_2=P\phi_2$, which by Lemma \ref{exotic1} lies in the class $OPS^m_{1,1}$. However, the same Lemma \ref{exotic1} also gives us the asymptotic expansion:
$$p_2(x,\xi)\sim \sum_{\abs{\alpha}\leq N} \frac{i^{\abs{\alpha}}}{\alpha!} D^\alpha_\xi p(x,\xi)D^\alpha_x \phi_2(x)$$
up to a term in $OPS^{m-N(1-\delta)}_{1,1}$. But then the symbol of $Q=\phi_1 P\phi_2$ is given by:
$$q(x,\xi)\sim \phi_1(x)\sum_{\abs{\alpha}\geq0} \frac{i^{\abs{\alpha}}}{\alpha!} D^\alpha_\xi p(x,\xi)D^\alpha_x \phi_2(x)$$
up to an operator in $S^{-\infty}$. Since $\phi_1, \phi_2$ have disjoint support, the above sum is zero, so it follows that $q\in S^{-\infty}$, as desired.
\end{proof}

The next lemma gives us a quantitative version of the smoothing property:
\begin{lemma}
\label{quantbound} Let $j\in\mathbb{Z}$ and let $P_j$ be a Littlewood-Paley projection operator. Let $\beta$ be an arbitrary real number, and assume that $\| f\|_{H^\beta}\leq 1$.
Let $Q$ be a pseudodifferential operator with symbol $q(x,\xi)$ in the class $S^m_{1,1}$. Let $\phi_1, \phi_2$ be two smooth functions with disjoint supports in the class $S^0_{1,\delta}$ where $1>\delta\geq0$. Then for any positive integer $N$, the composition $\phi_1Q\phi_2$ satisfies the following quantitative bound
$$\| P_j\phi_1 Q \phi_2 f\|_2 \leq K 2^{-jN}$$
with constant $K$ depending only on $N$ and $\beta$.
\end{lemma}
\begin{proof}
By Lemma \ref{exotic2}, we know that $\tilde{Q}:= \phi_1 Q \phi_2$ is a smoothing operator. In particular, we have for any real number $s$:
$$\| \tilde{Q} f\|_s \leq K \|f\|_\beta \leq K,$$
where the constant $K$ depends only on $s$ and $\beta$. By Lemma \ref{finiteband1}, we have
$$2^{sj} \|P_j \tilde{Q}f\|_2 \leq K \|\tilde{Q}f\|_s.$$
Together we get, for some other constant $K$ depending only on $s$ and $\beta$:
$$\|P_j \phi_1 Q\phi_2 f\|_2 \leq K 2^{-s j}.$$
Since $s$ was arbitrary, we achieve what we desired.
\end{proof}
\subsection{Littlewood-Paley Decomposition and Localization}
We use a Littlewood-Paley partition of frequency space. Namely, for $j\in \mathbb{Z}$ we have pseudodifferential operators $P_j$ with smooth symbols $p_j(\xi)$ that are supported in $\frac{2}{3}2^j < \abs{\xi}<3\cdot 2^j$, which also satisfy $p_j(\xi) = p_0(2^{-j} \xi)$ and 
\begin{equation}\sum_{j\in \mathbb{Z}} p_j(\xi) =1.\end{equation}
We define $\tilde{P}_j:= \sum_{k=-2}^{2} P_{j+k}$ to be the sum of all Littlewood-Paley projections whose symbols' supports intersect the support of $p_j(\xi)$. Likewise we may define $\tilde{p}_j(\xi)$. We also define $\tilde{\tilde{P_j}}:= \sum_{k=-4}^4 P_{j+k}$ to be the sum of all Littlewood-Paley projections whose symbols' supports intersect the support of $\tilde{p}_j$. Likewise we may define $\tilde{\tilde{p}}_j$.

We also note that
\begin{equation}
\label{proj1}
\tilde{P}_j P_j = P_j
\end{equation}
These Littlewood-Paley projections satisfy the following inequalities
\begin{lemma}
\label{PIneq}
For any $2\leq q \leq \infty$ there exists a constant $K$ independent of $j$ such that
$$\| P_j f\|_q \leq K 2^{nj\big(\frac{1}{2}-\frac{1}{q}\big)} \|P_j f\|_2$$
\end{lemma}
\begin{proof}
Let $\phi_j = \mathcal{F}^{-1}(\tilde{p}_j)$. Recalling the scaling law satisfied by the symbols $p_j$, we see that
\begin{equation}\label{scale1}\phi_j(x)=2^{nj}\phi_0(2^j x).\end{equation}
By definition, $\tilde{P}_jf = \phi_j \ast f$. By Young's convolution inequality (for which we need $q\geq 2$) and the projection property Equation (\ref{proj1}) we have
$$\| P_j f\|_q = \|\tilde{P}_j P_jf\|_q \leq \| P_j f\|_2 \|\phi_j \|_{2q/(2+q)}$$
By the scaling property in Equation (\ref{scale1}) and a change of variables we get:
$$\|\phi_j\|_{2q/(2+q)} =2^{nj} \|\phi_0(2^jx)\|_{2q/(2+q)}= 2^{nj(1/2-1/q)}\|\phi_0\|_{2q/(2+q)}.$$
Combining the last two inequalities yields our result.
\end{proof}
The next two lemmas may together be characterized as representative of the "finite band property" relating derivatives with Littlewood-Paley projections using the localization in frequency space. For more details on the finite band property, see \cite{KR} and Chapter 2 in \cite{BCD}. The following result is just a restatement of a lemma from \cite{BCD} in our notation.
\begin{lemma}[Lemma 2.1 in \cite{BCD}]
\label{finiteband1}
For any $s\geq 0$ there exists a constant $K$ such that for any $j\in \mathbb{Z}$ and for all $1\leq p\leq \infty$ we have
$$\|P_j f\|_{W^{s,p}(\mathbb{R}^n)} \leq K 2^{js} \|P_j f\|_{L^p(\mathbb{R}^n)}$$
$$2^{js} \|P_j f\|_{L^p(\mathbb{R}^n)} \leq K \|P_j f\|_{W^{s,p}(\mathbb{R}^n)}.$$
\end{lemma}
We shall need the following $L^\infty$ estimates on the high-frequency cutoff of each bump function $\phi_Q$.
\begin{lemma}
\label{highfreq}
Let $\phi=\phi_{Q,j}$ be a bump function of type $j$ (as in Section 2). Define $$\phi_2:= \mathcal{F}^{-1}(\chi_{\abs{\xi}>\frac{1}{100}2^j}(\xi)\mathcal{F}\phi(\xi)).$$ We call $\phi_2$ the high frequency cutoff of $\phi$. Then, given any $N$, there is a constant $K$ depending only on $N$ and $\epsilon$ such that:
$$\max(\|\phi_2(x)\|_\infty,\|\mathcal{F}(\phi_2)(\xi)\|_\infty) \leq K 2^{-jN}.$$
\end{lemma}
\begin{proof}
By the Schwartz-Paley-Wiener Theorem stated in Lemma \ref{SPW}, and since $\phi_2$ is smooth and compactly supported with Fourier transform $\mathcal{F}(\phi_2)$ supported away from the origin of frequency space, we may conclude that there is a constant $K$ depending only on $N$ such that
$$\| \mathcal{F}\phi_2(\xi)\|_\infty \leq K 2^{-jN}.$$
We prove a similar inequality for $\|\phi_2(x)\|_\infty$ by a rescaling argument. 
Let $N$ be arbitrary. Let $\psi$ be a smooth function with compact support such that $\textrm{supp}(\mathcal{F}(\psi))=\{\xi : \abs{\xi}\geq \frac{1}{100}\}$. We claim that if $\|D^\alpha \psi(x)\|_\infty\leq C_\alpha 2^{(-\epsilon\abs{\alpha}-n)j}$ for all $\alpha$, then $\| \psi(x)\|_\infty \leq K(N,\epsilon) 2^{(-N-n)j}$, where $K(N,\epsilon)$ depends only on $N$ and $\epsilon$ but not on $j$. Assuming the claim, let $\psi$ be a function satisfying the hypotheses of the claim and consider the rescaled function $$\tilde{\psi}(x):= 2^{nj}\psi(2^{j}x).$$ The Fourier transform of this function is
$$\mathcal{F}(\tilde{\psi})(\xi) = \mathcal{F}(\psi)(2^{-j}\xi).$$
The support of $\mathcal{F}(\tilde{\psi})$ is $\{\xi : \abs{\xi}\geq \frac{1}{100}2^{j}\}$, our claim yields
$$\|\tilde{\psi}(x)\|_\infty \leq 2^{-Nj},$$ and the following derivative bounds hold:
$$\abs{D^\alpha \tilde{\psi}(x)}\leq C_\alpha 2^{(\abs{\alpha}(1-\epsilon))j}.$$
Proving the $L^\infty$ bound for $\tilde{\psi}$ assuming the above derivative bounds and high frequency support is equivalent to proving the $L^\infty$ bound claimed before for $\psi$. We recognize $\phi_2$ as one such function $\tilde{\psi}$, so we just have to prove the statement regarding $\psi$ to finish the proof of the lemma.

Consider instead $f=2^{nj}\psi$, and assume to the contrary that there exists an $N$ such that for any $K$, we have
$$\|f\|_\infty >K2^{-Nj}\quad\textrm{and}\quad \|D^\alpha f\|_\infty \leq C_\alpha 2^{-\epsilon\abs{\alpha}j} \quad \forall \alpha>0 $$
However, the uniform derivative bounds preclude the arbitrary size of $f$, which is a contradiction.
\end{proof}

In the statement of the following lemmas, the constant $100$ appears. The particular value of this constant is unimportant and may be replaced by any large number of our choice. In addition, the hypothesis $k\geq j$ appears. This hypothesis can be weakened to $k\geq j-K$ for any given positive constant $K$ without any significant modification to the arguments used.
The next lemma is a commutator estimate similar to Proposition 5.2 in \cite{KP}.
\begin{lemma}
\label{commutator1}
Let $2<q<\infty$ be arbitrary. 
Suppose $\| f\|_2 \leq K$. Suppose $\phi$ is a bump function of type $j$ and let $k\geq j-2$ be arbitrary. Then, for a constant $K$ depending only on $q$:
$$\|\phi P_k f-\tilde{P}_k\phi P_k f\|_2\leq K2^{-j(100\cdot\lfloor \frac{q}{q-2}\rfloor)},$$
and
$$\|\phi \tilde{P}_k f-\tilde{\tilde{P}}_k\phi \tilde{P}_k f\|_2\leq K2^{-j(100\cdot\lfloor \frac{q}{q-2}\rfloor)}.$$
\end{lemma}
\begin{proof}
Given our bump function, we can decompose it into the sum $\phi = \phi_1+\phi_2$ such that $\mathcal{F}\phi_2(\xi)=\chi_{\abs{\xi}> \frac{1}{100}2^j}(\xi)\mathcal{F}\phi(\xi)$ for some $M$. By Lemma \ref{highfreq}, we know there is a constant $K$ depending only on $q$ such that:
$$\max\big(\|\phi_2(x)\|_\infty,\| \mathcal{F}\phi_2(\xi)\|_\infty\big) \leq K 2^{-j(100\cdot\lfloor \frac{q}{q-2}\rfloor)}.$$
On the other hand, we have:
\begin{equation}\label{YES}\tilde{P}_k\phi_1P_k f= \phi_1P_k f\end{equation}
and
\begin{equation}\label{YES2}\tilde{\tilde{P}}_k\phi_1\tilde{P}_k f= \phi_1\tilde{P}_k f.\end{equation}
Indeed, up to a dimensional constant we can write
$$\tilde{P}_k\phi_1 P_k f= \iint e^{ix\cdot \xi} \tilde{p}_k(\xi)\mathcal{F}(\phi_1)(z)p_k(\xi-z)\mathcal{F}(f)(\xi-z)dzd\xi$$
and
$$\tilde{\tilde{P}}_k\phi_1 \tilde{P}_k f= \iint e^{ix\cdot \xi} \tilde{\tilde{p}}_k(\xi)\mathcal{F}(\phi_1)(z)\tilde{p}_k(\xi-z)\mathcal{F}(f)(\xi-z)dzd\xi.$$
Now if $\xi-z$ is in the support of $p_k$ and $\abs{z}<\frac{1}{100}2^j$, then, since $k\geq j-2$, we have that $\xi$ remains well inside the support of $\tilde{p}_k$. It follows that the $\tilde{p}_k$ term in the integrand can be discounted, so Equation (\ref{YES}) is true. Likewise if $\xi-z$ is in the support of $\tilde{p}_k$ and $\abs{z}<\frac{1}{100}2^j$, then, since $k\geq j-2$, we have that $\xi$ remains well inside the support of $\tilde{\tilde{p}}_k$. It follows that the $\tilde{\tilde{p}}_k$ term in the integrand can be discounted, so Equation (\ref{YES2}) is true. Now observing that 
$$\phi P_k f- \tilde{P}_k \phi P_k f= \phi_1 P_k f-\tilde{P}_k \phi_1P_k f+\phi_2 P_k f-\tilde{P}_k\phi_2 P_k f$$ and that $\|f\|_2\leq K$, we conclude that
$$\|\phi P_k f- \tilde{P}_k \phi P_k f\|_2 = \|\phi_2 P_k f-\tilde{P}_k \phi_2 P_k f\|_2 \leq K\max(\|\phi_2(x)\|_\infty,\|\mathcal{F}(\phi_2)(\xi)\|_\infty)\leq$$ $$\leq K2^{-j(100\cdot\lfloor \frac{q}{q-2}\rfloor)}$$
Likewise we have
$$\phi \tilde{P}_k f- \tilde{\tilde{P}}_k \phi \tilde{P}_k f= \phi_1 \tilde{P}_k f-\tilde{\tilde{P}}_k \phi_1\tilde{P}_k f+\phi_2 \tilde{P}_k f-\tilde{\tilde{P}}_k\phi_2 \tilde{P}_k f$$ and that $\|f\|_2\leq K$, we conclude that
$$\|\phi \tilde{P}_k f- \tilde{\tilde{P}}_k \phi\tilde{ P}_k f\|_2 = \|\phi_2 \tilde{P}_k f-\tilde{\tilde{P}}_k \phi_2 \tilde{P}_k f\|_2 \leq K\max(\|\phi_2(x)\|_\infty,\|\mathcal{F}(\phi_2)(\xi)\|_\infty)\leq$$ $$\leq K2^{-j(100\cdot\lfloor \frac{q}{q-2}\rfloor)}$$
\end{proof}
One case of the next commutator estimate is used implicitly in \cite{KP}, although neither a statement nor a proof is given there.
\begin{lemma}
\label{infcommute}
Let $2<q<\infty$ be arbitrary. 
Suppose $\| f\|_2 \leq K$. Suppose $\phi$ is a bump function of type $j$ and let $k\geq j-2$ be arbitrary. Then, for a constant depending only on $q$:
$$\|\phi P_k f-\tilde{P}_k\phi P_k f\|_\infty\leq K2^{-j(100\cdot\lfloor \frac{q}{q-2}\rfloor)}.$$
\end{lemma}
\begin{proof}
As before, we decompose $\phi=\phi_1+\phi_2$ where $\phi_2$ is the high frequency cutoff. By Lemma \ref{highfreq}, where we let $N=(100\cdot\lfloor \frac{q}{q-2}\rfloor)$, there is a constant $K$ depending only on $q$ such that:
$$\max\big(\|\phi_2(x)\|_\infty,\| \mathcal{F}\phi_2(\xi)\|_\infty\big) \leq K 2^{-j(100\cdot\lfloor \frac{q}{q-2}\rfloor)}.$$
As in the proof of Lemma \ref{commutator1}, we just have to estimate:
$$\|\phi_2 P_k f-\tilde{P}_k\phi_2 P_k f\|_\infty.$$
Now using the fact that $\|f\|_{L^2}\leq K$, our $L^\infty$ bounds for $\phi_2$ and $\mathcal{F}(\phi_2)$, and Lemma \ref{PIneq}, we reach our desired conclusion.
\end{proof}
\begin{lemma}\label{GOODYTWO}
Let $\phi$ be a bump function of type $j$, let $k\geq j-2$ and $2<q<\infty$ be arbitrary. Suppose also that $\|f \|_2 \leq K$. Then:
$$\|\phi P_k f\|_\infty \leq K 2^{nk/2}\|\phi P_k f\|_2 +K 2^{-j(100\cdot\lfloor \frac{q}{q-2}\rfloor)}$$
\end{lemma}
\begin{proof}
This follows immediately from Lemma \ref{infcommute}, Lemma \ref{PIneq}, and the observation that
$$\phi P_k f= \phi P_k f -\tilde{P}_k\phi P_k f+ \tilde{P}_k\phi P_k f.$$
\end{proof}
The next lemma extends Lemma 5.3 in \cite{KP}.
\begin{lemma}
\label{ineq2}
Let $2\leq q \leq \infty$, $\|f\|_2\leq K$, and let $\phi$ be a bump function of type $j$. Let $k\geq j-2$ be arbitrary. Then:
$$\|\phi P_k f \|_q \leq K 2^{nk(1/2-1/q)}\|\phi P_k f\|_2 + K2^{-100j}$$
\end{lemma}
\begin{proof}
When $q=2$ there is nothing to show. When $q=\infty$, the result is proven in Lemma \ref{GOODYTWO}, and the exponent in the error term can be made arbitrarily negative. We prove the remaining cases by interpolation. First, it is evident from Lemma \ref{PIneq} that the function $\phi P_k f \in L^q$ for all $2\leq q \leq \infty$. In particular, since the $L^p$ norms are log-convex:
$$\|\phi P_k f\|_q \leq \|\phi P_k f\|_2^{2/q} \| \phi P_k f\|_\infty^{1-2/q}$$
For the remainder of the proof, we denote $A:= \|\phi P_k f\|_2$. Using Lemma \ref{GOODYTWO}, which is the case $q=\infty$ in the statement of our lemma, we have that:
$$\|\phi P_k f\|_q \leq A^{2/q}\bigg(2^{nk/2}A+ K2^{-100j\lfloor \frac{q}{q-2}\rfloor}\bigg)^{(q-2)/q}$$
Using the elementary inequality $(a+b)^p\leq 2^p(a^p+b^p)$ for $p\geq 0$ and the fact that $q>2$ gets us:
$$\|\phi P_k f\|_q \leq 2^{(q-2)/q} A^{2/q}\bigg(2^{\frac{nk}{2}\cdot\frac{q-2}{q}}A^{(q-2)/q}+ K2^{-100j}\bigg)$$
Using the fact that $\|f\|_2\leq K$ we conclude:
$$\|\phi P_k f\|_q \leq K 2^{nk(\frac{1}{2}-\frac{1}{q})}\|\phi P_k f\|_2 +K2^{-100j}$$
as desired.
\end{proof}
We can commute bump functions with Littlewood-Paley projections, as long as we multiply with another bump function with slightly larger support.
\begin{lemma}
\label{bumpcommute}
Let $k\geq j-2$, let $\phi_{Q,j}$ be a bump function at level $j$. Let $\beta<100$ and consider $f$ with $\| f\|_{H^\beta}\leq K$. Then
$$\|(1-\phi_{(1+2^{-\epsilon j})Q,j})P_k \phi_{Q,j} f\|_2 \leq K2^{-100j}$$
and
$$\|(1-\phi_{(1+2^{-\epsilon j})Q,j})\tilde{P}_k \phi_{Q,j} f\|_2 \leq K2^{-100j}.$$
\end{lemma}
\begin{proof}
By Lemma \ref{smoothing}, we see that
$$(1-\phi_{(1+2^{-\epsilon j})Q,j})P_k\phi_{Q,j}$$
is a smoothing pseudodifferential operator. By Lemma \ref{commutator1} and since $\tilde{P}_k P_k = P_k$, we have
$$(1-\phi_{(1+2^{-\epsilon j})Q,j})P_k\phi_{Q,j} = P_k\phi_{Q,j} - \phi_{(1+2^{-\epsilon j})Q,j}P_k\phi_{Q,j}  =$$ $$= \tilde{P}_kP_k\phi_{Q,j} - \tilde{P}_k\phi_{(1+2^{-\epsilon j})Q,j}P_k\phi_{Q,j} +R_j = \tilde{P}_k((1-\phi_{(1+2^{-\epsilon j})Q,j})P_k\phi_{Q,j}) +R_j,$$
where $R_j$ is an error term satisfying $\abs{R_j}\leq K2^{-200j}$.
By Lemma \ref{quantbound}, we have a quantitative bound on the Sobolev smoothing operator that appears, thus proving our desired inequality. By Lemma \ref{smoothing}, we see that
$$(1-\phi_{(1+2^{-\epsilon j})Q,j})\tilde{P}_k\phi_{Q,j}$$
is a smoothing pseudodifferential operator. By Lemma \ref{commutator1} and since $\tilde{\tilde{P}}_k \tilde{P}_k = \tilde{P}_k$, we have
$$(1-\phi_{(1+2^{-\epsilon j})Q,j})\tilde{P}_k\phi_{Q,j} = \tilde{P}_k\phi_{Q,j} - \phi_{(1+2^{-\epsilon j})Q,j}\tilde{P}_k\phi_{Q,j}  =$$ $$= \tilde{\tilde{P}}_k\tilde{P}_k\phi_{Q,j} - \tilde{\tilde{P}}_k\phi_{(1+2^{-\epsilon j})Q,j}\tilde{P}_k\phi_{Q,j} +R_j = \tilde{\tilde{P}}_k((1-\phi_{(1+2^{-\epsilon j})Q,j})\tilde{P}_k\phi_{Q,j}) +R_j,$$
where $R_j$ is an error term satisfying $\abs{R_j}\leq K2^{-200j}$.
By Lemma \ref{quantbound}, we have a quantitative bound on the Sobolev smoothing operator that appears, thus proving our desired inequality. The same proof works to get an arbitrarily negative exponent in the error term instead.
\end{proof}
Note that in the previous lemma, the order of the bump functions in the composition was unimportant. The only essential fact was disjointness of support.
\begin{lemma}
\label{collectbound}
Let $\mathcal{A}$ be a collection of cubes at level $j$ covering a measurable set $E$. For any $f\in L^2$, we have:
$$\|\chi_E P_j f\|_2^2 \leq \sum_{Q\in\mathcal{A}} f_Q^2$$
\end{lemma}
\begin{proof}
We notice that
$$\int \chi_E \abs{P_j f}^2 \leq \int \bigg(\sum_{Q\in\mathcal{A}} \phi_{Q,j}^2\bigg)\abs{P_j f}^2 =\sum_{Q\in\mathcal{A}}f_Q^2.$$ 
\end{proof}
\subsection{Hausdorff Measure and Dimension}
We denote a ball of radius $r$ in $\mathbb{R}^n$ by $B_r$.
Let $E\subset \mathbb{R}^n$. We recall that, up to a dimensional constant multiple, the $d$-dimensional Hausdorff measure of $E$ is given by
$$H^d(E) := \sup_{\delta >0} \inf_{\mathcal{C}_\delta(E)} \sum_{B_r\in \mathcal{C}_\delta(E)}r^d$$
where we have taken an infimum over all coverings $C_\delta(E)$ of $E$ by balls of radius less than or equal to $\delta$.
We define the Hausdorff dimension of $E$ to be:
$$\mathcal{H}(E) = \inf \{ d : H^d(E)=0\}.$$
We need a way to compute the Hausdorff dimension that follows from a discretization by sets of scale $2^j$. In particular we have the following lemma from \cite{KP}:
\begin{lemma}
\label{dimcompute}
Let $\mathcal{A}_j$ be a sequence of collections of balls in $\mathbb{R}^n$ so that each element of $\mathcal{A}_j$ has radius $2^{-j}$. Suppose that the number of balls in each $\mathcal{A}_j$, denoted $N(\mathcal{A}_j)$, is bounded by $N(\mathcal{A}_j)\leq C 2^{jd}$ where $C$ is independent of $j$. Let
$$E = \limsup_{j\to \infty}\mathcal{A}_j:= \cap_{j\in \mathbb{N}} \cup_{k>j} \cup_{B\in \mathcal{A}_k} B$$
be the set of points in infinitely many of the unions $\cup_{B\in\mathcal{A}_j} B$. Then $\mathcal{H}(E)\leq d$. \end{lemma}
\begin{proof}
By the definition of the Hausdorff dimension, it suffices to show that $H^s(E)=0$ for all $s>d$. To that end, pick $j$ large enough that $2^{-j} <\delta$. By construction of our set, $E$ can be covered by $\cup_{k>j} \cup_{B\in \mathcal{A}_k} B$. Evidently, we have
$$H^s(E)\leq \sum_{k>j} N(\mathcal{A}_k)(2^{-k})^s\leq C \sum_{k>j} 2^{kd}(2^{-k})^s $$
and the right-hand side above converges to zero as $j\to \infty$ so long as $d<s$.
\end{proof}
We also include in this subsection the standard Vitali covering lemma, which is often used in estimating the Hausdorff dimension of a set.
\begin{lemma}[Vitali]
\label{Vitali}
Let $\mathcal{A}$ be a collection of cubes. Then there is subcollection $\mathcal{A}'$ so that any two cubes in $\mathcal{A}'$ are pairwise disjoint and 
$$\bigcup_{Q\in \mathcal{A}} Q \subset \bigcup_{Q\in \mathcal{A}'} 5Q.$$
\end{lemma}
\subsection{Zero-Momentum Schwartz Vector Fields}
\begin{lemma}
\label{DivF}
A Schwartz function $\psi: \mathbb{R}^3 \to \mathbb{R}$ can be written as $\psi = \textrm{div } \Psi$ for some Schwartz vector field $\Psi$ if and only if $\int_{{\mathbb{R}^3}}\psi =0$.
\end{lemma}
\begin{proof}
One direction is relatively straightforward. Indeed suppose that $\psi = \textrm{div } \Psi$ for some Schwartz vector field and let $B_R$ be some ball centered at the origin of radius $R$. Then the divergence theorem gets us:
$$\int_{B_R} \psi = \int_{B_R} \textrm{div } \Psi = \int_{\partial B_R} \Psi \cdot \bf{n}$$
Since $\Psi$ is Schwartz and decays faster than any polynomial at infinity, the right hand side above tends to zero as $R\to \infty$ and the result is shown.\\\\
Now assume that $\int_{\mathbb{R}^3} \psi =0$; we desire to construct $\Psi$. Let $f(z)$ be any real-valued function in $C^\infty_c(\mathbb{R})$. Let 
$$\Gamma(x,y,z) = f(z)\cdot \int_{-\infty}^x \int_{-\infty}^y \int_{-\infty}^\infty \psi(r,s,t)dtdsdr$$
By our assumption on the integral of $\psi$, the function $\Gamma$ is a Schwartz function.
Now, 
$$\partial_x \Gamma = f(z)\int_{-\infty}^y \int_{-\infty}^\infty \psi(x,s,t)dtds$$
and
$$\partial_y \Gamma = f(z)\int_{-\infty}^x \int_{-\infty}^\infty \psi(r,y,t)dtdr.$$
So defining (a function that is also a Schwartz function by our assumption on $\psi$):
$$\Xi(x,y,z) = \int_{-\infty}^z (\psi - \partial_x \Gamma -\partial_y \Gamma)(x,y,t)dt$$
we observe that
$$\psi = \partial_x \Gamma + \partial_y \Gamma +\partial_z \Xi$$
so our desired vector field $\Psi = (\Gamma, \Gamma, \Xi)$.
\end{proof}

\subsection{The Continuity of the Energy Integral}
\begin{lemma}\label{Integral1}
Let $u$ be any smooth solution of Equation (\ref{PEQN}) with Schwartz initial value $u_0$, let $Q$ be any cube at level $j$, and let $t_0<T$ be any time before blowup. Then the function of time $u_Q(t)$ is continuous on $[0,T)$ and the following equality holds:
$$\lim_{\epsilon\to 0} \int_{t_0}^{T-\epsilon} \frac{d}{dt} u_{Q}^2 dt = \int_{t_0}^T \frac{d}{dt}u_{Q}^2 dt $$
\end{lemma}
\begin{proof}
The vector field $u(t)$ is a smooth (mild) solution of Equation (\ref{PEQN}),
$$\partial_t u +(-\Delta)^\alpha u +\hat{B}(u,u) =0,$$
with Schwartz initial data $u_0$. In particular, its $L^2$ norm is strongly continuous in time, so it remains to show that the localized norm $u_Q$ remains continuous in time. For this we proceed as in the proof of Lemma 2.7 in \cite{O}.

We observe that
$$\| \phi_{Q,j}P_ju(x,t) - \phi_{Q,j}P_ju(x,s) \|_{L^2(\mathbb{R}^n)}^2 \leq \int_{2Q} \abs{\int_{\mathbb{R}^n} p_j(x-\xi)(\mathcal{F}(u)(\xi,t)-\mathcal{F}(u)(\xi,s))d\xi}^2dx.$$
The integral inside the absolute value on the right-hand side above converges to zero because $u$ is a smooth mild solution whose $L^2$ norm is strongly continuous in time and because the Fourier transform is a unitary operator on $L^2$. Moreover, since $\|u\|_{L^2}\leq K$, the integral is dominated by:
$$\abs{\int_{\mathbb{R}^n} p_j(x-\xi)(\mathcal{F}(u)(\xi,t)-\mathcal{F}(u)(\xi,s))d\xi}\leq K \tilde{p}_j(x)$$
Thus, by the dominated convergence theorem, we conclude that the localization $u_Q$ remains strongly continuous in time. The desired limit then follows in a straightforward manner.
\end{proof}
As a corollary, we have:
\begin{corollary}\label{Integral2}
Let $t_0<T$ be any time. Then 
$$ \int_{t_0}^{T} \frac{d}{dt} u_Q^2 dt = \lim_{t\to T} u_Q^2(t) -u_Q^2(t_0).$$
\end{corollary}
\begin{proof}
Indeed, we may write
$$ \int_{t_0}^{T-\epsilon} \frac{d}{dt} u_Q^2 dt = u_Q^2(T-\epsilon) -u_Q^2(t_0)$$
and take the limit of the above expression as $\epsilon\to 0$, using Lemma \ref{Integral1}.
\end{proof}

\subsection{The Critical Exponent in Higher Dimensions}
We determine the critical exponent of hyperdissipation that guarantees global regularity for solutions of Equation (\ref{FRACNS}). A proof along the same lines for the hyperdissipative three-dimensional Navier-Stokes equations may be found in \cite{KP}.

\begin{proposition}
Let $\alpha>(n+2)/4$ and let $u(t):\mathbb{R}^n\times [0,\epsilon)\to \mathbb{R}^n$ be a smooth solution of Equation (\ref{FRACNS}) from Schwartz initial data. Then, $u(t)$ is smooth for all times $t\geq 0$. 
\end{proposition}
\begin{proof}
We let $\beta>0$ and test Equation (\ref{FRACNS}) against $(-\Delta)^\beta u$ to get:
$$\langle \partial_t u , (-\Delta)^\beta u \rangle +\langle (-\Delta)^\alpha u, (-\Delta)^\beta u\rangle +\langle (u\cdot\nabla)u, (-\Delta)^\beta u\rangle =0.$$
Using that $\textrm{div }u=0$, we may rewrite the equation as:
\begin{equation}\label{UH1} \frac{d}{dt}\big(\frac{1}{2}\|u\|_{\beta}^2\big)+\|u\|_{\alpha+\beta}^2 +\langle (-\Delta)^{\frac{\beta}{2}}\textrm{div}(u\otimes u), (-\Delta)^{\frac{\beta}{2}} u\rangle =0.\end{equation}
We estimate the trilinear term by:
\begin{equation}\label{UH2}\abs{\langle (-\Delta)^{\frac{\beta}{2}}\textrm{div}(u\otimes u), (-\Delta)^{\frac{\beta}{2}} u\rangle}\leq \| u\|_{W^{1,p}}\|u\|_{W^{\beta,q}}\|u\|_\beta\end{equation}
where $1/p+1/q=1/2$. Now, $\alpha=(n+2)/4$ is the minimum value of $\alpha$ so that both inequalities 
$$\frac{1}{2}-\frac{\alpha}{n}\leq \frac{1}{p}-\frac{1}{n}, \quad \textrm{and} \quad \frac{1}{2}-\frac{\alpha}{n}\leq \frac{1}{q}$$
hold for some choice $(p,q)$ with $1/p+1/q=1/2$ depending on $n\geq 3$. Since these inequalities hold, we may use the Sobolev Embedding Theorem and then Young's inequality to conclude that:
\begin{equation}\label{UH3} \| u\|_{W^{1,p}}\|u\|_{W^{\beta,q}}\|u\|_\beta\leq \|u\|_\alpha \|u\|_{\alpha+\beta}\|u\|_\beta\leq \frac{1}{2} \|u\|_{\alpha+\beta}^2 +\frac{1}{2}\|u\|_\alpha^2\|u\|_\beta^2.\end{equation}
Using Equations (\ref{UH1}), (\ref{UH2}), and (\ref{UH3}), we get
$$\frac{d}{dt}\big(\frac{1}{2}\|u\|_{\beta}^2\big) \leq \frac{1}{2}\|u\|_\alpha^2\|u\|_\beta^2$$
It follows that
$$\|u(t)\|_\beta^2 \leq \|u(0)\|_\beta^2 \exp\big(\int_0^t \|u(s)\|_\alpha^2ds\big).$$
However, by the energy dissipation law, we know
$$\int_0^t \|u(s)\|_\alpha^2 ds\leq \|u(0)\|_{L^2}$$
for all times $t$ whenever $u(t)$ is smooth. Since $\beta$ was arbitrary, $u_0$ was Schwartz, and $u(t)$ was smooth for small times $t\in [0,\epsilon)$, we get regularity for all times $t\geq 0$.
\end{proof}

\end{document}